\documentclass[preprint,12pt]{elsarticle}

\usepackage{amssymb}
\usepackage{amsmath,amscd,amsthm}

\numberwithin{equation}{section}
\newtheorem{theorem}{Theorem}[section]
\newtheorem{lemma}[theorem]{Lemma}

\newtheorem{proposition}[theorem]{Proposition}

\newtheorem{definition}[theorem]{Definition}

\def\X{{\mathrm{X}}}

\def\R{\mathbb R}

\def\H{\mathbb H}

\begin{document}

\begin{frontmatter}



\title
{Lower bound of Riesz transform kernels and Commutator Theorems on stratified nilpotent\\ Lie groups}


\author[a]{Xuan Thinh Duong\fnref{fn1}}
\ead{xuan.duong@mq.edu.au}

\author[b]{Hong-Quan Li\corref{cor1}\fnref{fn2}}
\ead{hongquan\_li@fudan.edu.cn}
\cortext[cor1]{Corresponding author}

\author[a]{Ji Li\corref{}\fnref{fn1}}
\ead{ji.li@mq.edu.au}

\author[c]{Brett D. Wick\corref{}\fnref{fn3}}
\ead{wick@math.wustl.edu}

\address[a]{Department of Mathematics, Macquarie University, NSW, 2109, Australia}
\address[b]{School of Mathematical Sciences, Fudan University, 220 Handan Road, Shanghai 200433, People's Republic of China}
\address[c]{Department of Mathematics,
         Washington University - St. Louis,
         St. Louis, MO 63130-4899 USA}

\fntext[fn1]{Supported by ARC DP 160100153.}
\fntext[fn2]{Partially supported by NSF of China (Grants No. 11625102 and No. 11571077) and ``The
Program for Professor of Special Appointment (Eastern Scholar) at Shanghai Institutions
of Higher Learning''.}
\fntext[fn3]{Partially supported by National Science Foundation -- DMS \# 1560955.}

\begin{abstract}
We provide a study of the Riesz transforms on stratified nilpotent Lie groups, and obtain a certain version of the pointwise lower bound of the Riesz transform kernel. Then we establish the characterisation of the BMO space on  stratified nilpotent Lie groups via the boundedness of the commutator of the Riesz transforms and the BMO function. This extends the well-known Coifman, Rochberg, Weiss theorem from Euclidean space to the setting of stratified nilpotent Lie groups.  In particular, these results apply to the well-known example of the Heisenberg group. As an application, we also study the curl operator on the Heisenberg group and stratified nilpotent Lie groups, and establish the div-curl lemma with respect to the Hardy space on  stratified nilpotent Lie groups.

\end{abstract}


\end{frontmatter}




\section{Introduction and statement of main results}
\setcounter{equation}{0}

The aim of this paper is to study the lower bound of the kernel of Riesz transform on stratified nilpotent Lie groups, and
to establish the characterisation of the BMO space via the commutators of Riesz transforms and BMO functions. As an application, we also study the div-curl lemma on stratified nilpotent Lie groups.

Recall that the classical Hardy space $H^p$, $0<p<\infty$, on the unit disc
$\mathbb{D}=\{ z\in\mathbb{C}:\ |z|<1\}$ of the complex plane is defined as the space of holomorphic functions that satisfy
$$  \|f\|_{H^p(\mathbb{D})}:=\sup_{0\leq r<1} \Big( {1\over 2\pi}\int_0^{2\pi}  \big|f(re^{it})\big|^p dt \Big)^{1\over p}<\infty.  $$
It is well-known that the product of two
 $H^2(\mathbb{D})$ functions belongs to Hardy space $H^1(\mathbb{D})$, but in fact the converse is also true and  known as
 {\it  the Riesz factorization theorem}:
 {\it ``A function $f$ is in $H^1(\mathbb{D})$ if and only if there exist $g,h\in H^2(\mathbb{D})$ with
 $f= g\cdot h$ and $\|f\|_{H^1(\mathbb{D})} =\|g\|_{H^2(\mathbb{D})}\|h\|_{H^2(\mathbb{D})}$.''
  }
The same result holds for the  Hardy space $H^1$ on the unit circle $\mathbb {T}$. This factorisation result plays an important role in obtaining the following theorem:
 $$ \|[b,H]\|_{2\to2} \approx \|b\|_{\rm BMO(\mathbb T)}, $$
 where $[b,H](f)=b H(f) - H(bf)$ is the commutator of the Hilbert transform $H$ and a function $b$.  An interpretation of this result can be given in the language of Hankel operators and then can recover a famous result of Nehari, \cite{N}.
 We refer to \cite{L} for the history and literature of the Nehari theorem.

The real-variable Hardy space theory on
$n$-dimensional Euclidean space $\R^n$ ($n\geq1$) play an important
role in harmonic analysis and has been systematically developed.
However, the analogue of the Riesz factorisation theorem, sometimes referred to as strong factorisation,
is not true for real-variable Hardy space $H^1(\mathbb{R}^n)$.  Nevertheless,
Coifman, Rochberg and Weiss \cite{CRW} established a weak factorisation of the Hardy space $H^1(\mathbb{R}^n)$ via a bilinear form related to the Riesz transform (Hilbert transform in dimension 1).  This factorisation follows via functional analysis and duality from the characterisation of the space ${\rm BMO}(\R^n)$
via the commutators of the Riesz transforms and a function $b$, i.e.,
$$\sum_{j=1}^n \|[b,R_j]:L^2(\R^n)\to L^2(\R^n)\| \approx \|b\|_{{\rm BMO}(\R^n)}.$$
See also  \cite{U} for a direct proof of the weak factorisation. The proof of the lower bound of the commutator here relies on the
spherical harmonic expansion of the Riesz kernel.

There are quite a few recent developments on commutator results in different settings.  For example, in \cite{HLW} they obtained the characterisation of weighted BMO space via the two weight commutators of Riesz transforms. See also \cite{DHLWY} for the Riesz transforms associated with Neumann Laplacian  and \cite{DLW} for the Ahlfors--Beurling operators. In these  results listed above, to obtain the lower bound, they used the  expansion of the Riesz kernels, which relies on Fourier transforms.

In \cite{Ler} and \cite{GLW}, they considered the class of rough homogeneous singular integrals, which is of the  form
$T(f)(x) = \int_{\R^n}{\Omega(x-y)}|x-y|^{-n} f(y)dy$, and obtained the characterisation of weighted BMO space via the
two weight commutator of $T$, by assuming certain conditions on the homogeneous function $\Omega(x)$ of degree zero.
The proof of the lower bound depends on the specific form of the kernel and the assumptions on $\Omega$.

In \cite{DLWY} and \cite{LNWW}, they studied the commutators of the Bessel Riesz transforms and of the Cauchy integral operator respectively, where the Fourier transform expansion of the kernels is not available and those kernels are not of the homogeneous type as studied above. They obtained the lower bound of the commutator by constructing a suitable concrete weak factorisation of the appropriate Hardy spaces according to the specific expression of the kernel of the operators, and then proved that this weak factorisation implied the lower bound for the appropriate BMO norm.

\medskip

Inspired by the results above, it is natural to ask
whether this result holds on the Heisenberg groups $\mathbb H^n$. Note that in several complex variables,
the Heisenberg groups $\mathbb H^n$ is the boundary of the Siegel upper half space, whose roles are holomorphically equivalent to the unit sphere and the unit ball in $\mathbb C^n$. To be more specific, we aim to study the following question:  Does one have
\begin{align}\label{Question}
 \|[b,R_j]: {L^2(\mathbb H^n)\to L^2(\mathbb H^n)}\| \approx \|b\|_{{\rm BMO}(\mathbb H^n)} ?
\end{align}
Here $R_j$ is the $j$th Riesz transform on $\mathbb H^n$ and ${\rm BMO}(\mathbb H^n)$ is the BMO space
defined and studied in Folland--Stein \cite{FoSt}.  We provide an affirmative answer to this question, and in fact we prove this theorem in the general setting of stratified nilpotent Lie groups $\mathcal G$.

We point out that for general stratified nilpotent Lie groups,
there is no explicit kernel expression for the heat semigroup $e^{-tL}$ generated by the sub-Laplacian $L$ on $\mathcal G$.
Hence there is no explicit kernel expression for the Riesz transforms $R_j$ on $\mathcal G$. Thus, the method of spherical harmonic expansion (such as in \cite{CRW, HLW}) does not apply, and the method for rough homogeneous singular integrals (such as in \cite{GLW,Ler}) does not apply either. To obtain our theorem on $\mathcal G$, we  use the approach in \cite{LW}, which was inspired by \cite{U}. To achieve this, we establish a suitable version of pointwise kernel lower bound of the Riesz transform $R_j$ on $\mathcal G$, which is new  and should be useful for other  problems.

To be more specific, suppose $\mathcal G$ is a stratified nilpotent Lie group. Let $\{\X_j\}_{1 \leq j \leq n}$
be a basis for the left-invariant vector fields of degree one on $\mathcal G$.
Let $\Delta = \sum_{j = 1}^n \X_j^2 $ be the sub-Laplacian
on $\mathcal G$. Consider the $j^{\mathrm{th}}$  Riesz transform on $\mathcal G$ which is defined as $R_j:= \X_j (-\Delta)^{-\frac{1}{2}}$.

It is well-known that the Riesz transform $R_j$ is a Calder\'on--Zygmund operator on $\mathcal G$, i.e.,
it is  bounded on $L^2(\mathcal G)$ and the kernel satisfies the corresponding size and smoothness conditions, see for example \cite{CG,FoSt}. Moreover,
it is also well-known that from \cite{CG}, the set of identity operator together with Riesz transforms $\{I, R_1, R_2, \ldots, R_n\}$ characterises the Hardy space
$H^1(\mathcal G)$ introduced and studied by Folland and Stein \cite{FoSt}.

We now recall the BMO space on $\mathcal G$, which is the dual space of $H^1(\mathcal G)$ \cite[Chapter 5]{FoSt}, defined as
$$ {\rm BMO}(\mathcal G):=\{ b\in L^1_{loc}(\mathcal G):\  \|b\|_{{\rm BMO}(\mathcal G)}<\infty \},$$
where
$$\|b\|_{{\rm BMO}(\mathcal G)}:=\sup_B {1\over |B|}\int_B|b(g)-b_B|dg. $$
and $b_B:={1\over |B|}\int_Bb(g)\,dg$, where $B$ denotes the ball on $\mathcal G$ defined via a homogeneous norm $\rho$, see Section 2 for details.

The main results of this paper are twofold. First, we study the behaviour of the
lower bound of the kernel $K_j(g_1,g_2)$ of Riesz transform $R_j$ when $g_1$ and $g_2$ are
far away. We show that $K_j(g_1,g_2)$ satisfies a suitable version
of a non-degenerated condition (see more details of all the notation in Section 2 and Section 3 below), which is of independent interest and
will be useful in other problems.
\begin{theorem}\label{thm0}
Suppose that $\mathcal G$ is a stratified nilpotent Lie group with homogeneous dimension $Q$ and that $j=1,2,\ldots,n$. There exist constants $0 < \varepsilon_o \ll 1$ and $C(j,n)$ which depends on $j$ and $n$ only, such that for all $0<\eta\leq\varepsilon_o$, and for every $g\in \mathcal G$ there exist $g_*=g_*(j,g,r)\in\mathcal G$ with $\rho(g,g_*)=r$, the Riesz transform $R_j$ satisfies
  $$ \big|R_j\big(\chi_{B(g_*,\eta r)}\big)(g)\big|\geq C(j,n)\eta^Q, $$
  where $\chi_{B(g_*,\eta r)}$ is the characteristic function of the ball $B(g_*,\eta r)$ centered at $g_*$ with radius $\eta r$, and $Q$ is the homogeneous dimension of $\mathcal G$.
\end{theorem}

Second, based on the property of the Riesz transform kernel, we establish the following commutator theorem
on stratified nilpotent Lie groups $\mathcal G$ via constructing a suitable version of weak factorisation of the Hardy space $H^1(\mathcal G)$, which covers the setting of Heisenberg groups and hence give a positive answer to the question in \eqref{Question}.
\begin{theorem}\label{thm1}
Suppose that $\mathcal G$ is a stratified nilpotent Lie group and that $1<p<\infty$ and $j=1,2,\ldots,n$. Then the commutator of $b\in {\rm BMO}(\mathcal G)$ and the Riesz transform  $R_j$ satisfies
  $$ \|[b,R_j]: {L^p(\mathcal G)\to L^p(\mathcal G)}\| \approx \|b\|_{{\rm BMO}(\mathcal G)}, $$
  where the implicit constants are independent of the function $b$.
\end{theorem}

As an application, we study the curl operator on the stratified nilpotent Lie groups $\mathcal G$ and establish the
div-curl lemma with respect to the Folland--Stein Hardy space $H^1(\mathcal G)$. Note that the first version of div-curl lemma related to the Hardy space $H^1(\mathbb R^n)$ was due to \cite{CLMS}.
\begin{theorem}\label{thm2}
Suppose that $\mathcal G$ is a stratified nilpotent Lie group and that $E,B\in L^2(\mathcal G,\R^{n})$ are vector fields on $\mathcal G$ taking values in $\R^{n}$ and satisfy
 $$ \operatorname{div}_{\mathcal G} E=0\quad {\rm and}\quad \operatorname{curl}_{\mathcal G} B=0. $$
 Then we have $E\cdot B \in H^1(\mathcal G)$ with
 $$ \| E\cdot B \|_{H^1(\mathcal G)}\lesssim \|E\|_{L^2(\mathcal G,\R^{n})}\|B\|_{L^2(\mathcal G,\R^{n})}. $$
 \end{theorem}
For the details of the definition of ${\rm curl}_{\mathcal G}$ and the proof, we refer to Section~5.

This paper is organised as follows. In Section 2 we recall necessary preliminaries on stratified nilpotent Lie groups $\mathcal G$ and Heisenberg groups $\mathbb H^n$. Theorem \ref{thm0} will be proved in Section 3, where we establish a suitable version of lower bound of the kernel of
Riesz transforms associated with the sub-Laplacian on $\mathcal G$. Then in { Section 4}, we give a weak factorisation of the Hardy space $H^1(\mathcal G)$ and then provide the proof of Theorem \ref{thm1}. In the last section we
study the curl operators and then provide the proof of Theorem \ref{thm2}.

\section{Preliminaries on stratified nilpotent Lie groups $\mathcal G$ and on Heisenberg groups $\mathbb H^n$}
\setcounter{equation}{0}

Recall that a connected, simply connected nilpotent Lie group $\mathcal{G}$ is said to be stratified if its left-invariant Lie algebra $\mathfrak{g}$ (assumed real and of finite dimension) admits a direct sum decomposition
\begin{align*}
\mathfrak{g} = \bigoplus_{i = 1}^k V_i \  \mbox{where $[V_1, V_i] = V_{i + 1}$ for $i \leq k - 1$.}
\end{align*}
One identifies $\mathfrak{g}$ and $\mathcal{G}$ via the exponential map
\begin{align*}
\exp: \mathfrak{g} \longrightarrow \mathcal{G},
\end{align*}
which is a diffeomorphism.

We fix once and for all a (bi-invariant) Haar measure $dx$ on $\mathcal G$ (which is just the lift of Lebesgue measure on $\frak g$ via $\exp$).

There is a natural family of dilations on $\frak g$ defined for $r > 0$ as follows:
\begin{align*}
\delta_r \left( \sum_{i = 1}^k v_i \right) = \sum_{i = 1}^k r^i v_i, \quad \mbox{with $v_i \in V_i$}.
\end{align*}
This allows the definition of dilation on $\mathcal{G}$, which we still denote by $\delta_r$.

We choose once and for all a basis $\{\X_1, \cdots, \X_n\}$ for $V_1$ and consider the sub-Laplacian $\Delta = \sum_{j=1 }^n \X_j^2 $. Observe that $\X_j$ ($1 \leq j \leq n$) is homogeneous of degree $1$ and $\Delta$ of degree $2$ with respect to the dilations in the sense that:
\begin{align*}
&\X_j \left( f \circ \delta_r \right) = r \, \left( \X_j f \right) \circ \delta_r, \qquad  1 \leq j \leq  n, \  r > 0, \ f \in C^1, \\
&\delta_{\frac{1}{r}} \circ \Delta \circ \delta_r = r^2 \, \Delta, \quad \forall r > 0.
\end{align*}

Let $Q$ denote the homogeneous dimension of $\mathcal{G}$, namely,
\begin{align}\label{homo dimension}
Q= \sum_{i=1}^k i\, {\rm dim} V_i.
\end{align}
And let $p_h$ ($h > 0$)  be the heat kernel (that is, the integral kernel of $e^{h \Delta}$)
on $\mathcal G$. For convenience, we set $p_h(g) = p_h(g, o)$ (that is, in this article, for a convolution operator, we will identify the integral kernel with the convolution kernel) and $p(g) = p_1(g)$.

Recall that (c.f. for example \cite{FoSt})
\begin{align} \label{hkp1}
p_h(g) = h^{-\frac{Q}{2}} p(\delta_{\frac{1}{\sqrt{h}}}(g)), \qquad  \forall h > 0, \  g \in \mathcal G.
\end{align}

The kernel of the $j^{\mathrm{th}}$  Riesz transform $\X_j (-\Delta)^{-\frac{1}{2}}$ ($1 \leq j \leq  n$) is written simply as $K_j(g, g') = K_j(g'^{-1} \circ g)$. It is well-known that
\begin{align} \label{kjs}
K_j \in C^{\infty}(\mathcal G \setminus \{o\}), \ K_j(\delta_r(g)) = r^{-Q} K_j(g), \quad \forall g \neq o, \ r > 0, \ 1 \leq j \leq  n
\end{align}
which also can be explained by \eqref{hkp1} and the fact that
\begin{align*}
K_j(g) = \frac{1}{\sqrt{\pi}} \int_0^{+\infty} h^{-\frac{1}{2}} \X_j p_h(g) \, dh  = \frac{1}{\sqrt{\pi}} \int_0^{+\infty} h^{- \frac{Q}{2} - 1} \left( \X_j p \right)(\delta_{\frac{1}{\sqrt{h}}}(g)) \, dh.
\end{align*}

Next we recall the homogeneous norm $\rho$ (see for example \cite{FoSt}) on $\mathcal G$
defined to be a continuous function
$g\to \rho(g)$ from $\mathcal G$ to $[0,\infty)$, which is $C^\infty$ on $\mathcal G\backslash \{o\}$
and satisfies
\begin{enumerate}
\item[(a)] $\rho(g^{-1}) =\rho(g)$;
\item[(b)] $\rho({ \delta_r(g)}) =r\rho(g)$ for all $g\in \mathcal G$ and $r>0$;
\item[(c)] $\rho(g) =0$ if and only if $g=o$.
\end{enumerate}
For the existence (also the construction) of the homogeneous norm $\rho$ on  $\mathcal G$,
we refer to  \cite[Chapter 1, Section A]{FoSt}. For convenience,
we set
\begin{align*}
\rho(g, g') = \rho(g'^{-1} \circ g) = \rho(g^{-1} \circ g'), \quad \forall g, g' \in \mathcal{G}.
\end{align*}
Recall that (see \cite{FoSt}) this defines a quasi-distance in sense of  Coifman-Weiss, namely, there exists a constant $C \geq1 $ such that
 \begin{align} \label{qdr}
 \rho(g_1, g_2) \leq C \, \left( \rho(g_1, g') + \rho(g', g_2)   \right), \qquad \forall g_1, g_2, g' \in  \mathcal{G}.
\end{align}
In the sequel, we fixe a homogeneous norm $\rho$ on  $\mathcal G$.

We now denote by $d$ the Carnot--Carath\'eodory metric associated to the vector fields $\{\X_j\}_{1\leq j\leq n}$, which is equivalent to $\rho$ in the sense that: there exist $C_1>0$ and $C_2\geq1$ such that
for every $g_1,g_2\in\mathcal G$ (see \cite{BLU}),
\begin{align} \label{equi metric d}
 C_1\rho(g_1,g_2)\leq d(g_1,g_2)\leq C_2 \rho(g_1,g_2).
\end{align}
We point out that the Carnot--Carath\'eodory metric $d$ even on the most special stratified Lie group, the Heisenberg group, is not smooth on $\mathcal G \setminus \{o\}$.

In the sequel, for $g\in\mathcal G$ and $r>0$, {$B(g,r)$
denotes the open ball defined by $\rho$.
Without} loss of generality, we may suppose that the Haar measure has been normalized such that the measure of $B(o, 1)$, $|B(o, 1)| = 1$.

\medskip
We also recall the definition for Heisenberg group. Recall that $\mathbb{H}^{n}$ is
the Lie group with underlying manifold $\mathbb{C}^{n} \times \mathbb{R}= \{[z,t]: z \in \mathbb{C}^{n}, t \in \mathbb{R}\}$
and multiplication law
\begin{align*}
[z,t]\circ [z', t']&=[z_{1},\cdots,z_{n}, t] \circ [z_{1}',\cdots,z_{n}', t']\\
&: =\Big[z_{1}+z_{1}',\cdots, z_{n}+z_{n}',
t+t'+2 \mbox{Im} \big(\sum_{j=1}^{n}z_{j} \bar{z}'_{j}\big) \Big].
\end{align*}
The identity of $\mathbb{H}^{n}$ is the origin and the inverse is given
by $[z, t]^{-1} = [-z,-t]$.
Hereafter, we agree to identify $\mathbb{C}^{n}$  with $\mathbb{R}^{2n}$ and to use the following notation to denote the
points  of $\mathbb{C}^{n} \times \mathbb{R} \equiv \mathbb{R}^{2n+1}$:
$g=[z,t] \equiv [x,y,t] = [x_{1},\cdots, x_{n}, y_{1},\cdots, y_{n},t]
$
with $z = [z_{1},\cdots, z_{n}]$, $z_{j} =x_{j}+iy_{j}$ and $x_{j}, y_{j}, t \in \mathbb{R}$ for $j=1,\ldots,n$. Then, the
composition law $\circ$ can be explicitly written as
\begin{equation*}
g\circ g'=[x,y,t] \circ [x',y',t'] = [x+x', y+y', t+t' + 2\langle y, x'\rangle -2\langle x, y'\rangle],
\end{equation*}
where $\langle \cdot, \cdot\rangle$ denotes the usual inner product in $\mathbb{R}^{n}$.

The Lie algebra of the left invariant vector fields of $\H^n$ is generated by (here and in the following, we shall identify vector fields and the associated first order differential operators)
\begin{align}\label{XYT}
X_j={\partial \over \partial x_j} + 2y_j {\partial\over\partial t}, \quad Y_j={\partial \over \partial y_j} - 2x_j {\partial\over\partial t}, \quad T={\partial\over \partial t},
\end{align}
for $j=1,\ldots,n$. We now denote $X_{n+j}=Y_j$ for $j=1,\ldots,n$.

\vspace{0.2cm}

{\bf Notation:} In what follows, $C$, $C'$, etc. will denote various constants which depend only on the triple
$(\mathcal{G}, \rho, \{ \X_j \}_{1 \leq j \leq n})$, and on $p$ in Section \ref{S4}. By $A\lesssim B$, we shall mean $A\le C B$ with such a constant $C$, and similarly $A\approx B$ stands for both $A \le C B$  and  $B\le C A$.

\section{Lower bound for kernel of Riesz transform $R_j:= \X_j (-\Delta)^{-\frac{1}{2}}$ and proof of Theorem \ref{thm0}}
\setcounter{equation}{0}

In this section, we establish two versions of the lower bound for kernel of Riesz transform $R_j:= \X_j (-\Delta)^{-\frac{1}{2}}$, $j=1,\ldots,n$, on stratified nilpotent Lie groups $\mathcal G$, see Propositions \ref{lem lower bound} and \ref{lem lower bound 2}. These two versions of kernel lower bound are of independent interest and will be useful in studying other related problems. We point out that Theorem \ref{thm0} follows directly from Proposition \ref{lem lower bound 2}.

Here we will use the Carnot--{Carath\'eodory} metric $d$ associated to $\{\X_j\}_{1\leq j\leq n}$ to study the lower bound, and we also  make good use of the dilation structure on $\mathcal G$. It is not clear whether one can obtain similar lower bounds for the Riesz kernel on general nilpotent Lie groups which is not stratified.

To begin with, we first recall that
by the classical estimates for heat kernel and its derivations on stratified groups (see for example \cite{VSC92}), it is well-known that
\begin{align} \label{MEHT}
&|K_j(g, g')| + d(g, g') \sum_{i = 1}^n \left( |\X_{i, g} K_j(g, g')| + |\X_{i, g'} K_j(g, g')| \right) \nonumber \\
&\lesssim d(g, g')^{-Q}, \quad \forall 1 \leq j \leq n, \  g \neq g',
\end{align}
where $\X_{i, g}$ denotes the derivation with respect to $g$.

Next we have the following result showing that the Riesz kernel $K_j$ is not identically zero.
\begin{lemma}\label{lem non zero}
For all $1 \leq j \leq n$, we have $K_j \not\equiv 0$ in $\mathcal{G} \setminus \{ o \}$.
\end{lemma}

\begin{proof}
We argue by contradiction. Fix
\begin{align*}
0 \leq \phi_* \in C_0^{\infty}(B(o, 1)) \  \mbox{ with } \int \phi_*(g) \, dg = 1.
\end{align*}
By our assumption, we have for all $r$ large enough
\begin{align*}
0 &= \int_{B(o, 1)} K_j(g'^{-1} \circ \exp{(r \X_j)}) \phi_*(g') \, dg' \\
&= \left( R_j \phi_* \right)(\exp{(r \X_j)})   = \left\{ \X_j \left[ (-\Delta)^{-\frac{1}{2}} \phi_* \right] \right\}(\exp{(r \X_j)}) \\
&= \frac{d}{dr} \left\{ \left[ (-\Delta)^{-\frac{1}{2}} \phi_* \right](\exp{(r \X_j)}) \right\}.
\end{align*}
Then for a suitable $r_0 \gg 1$,
\begin{align*}
0 &= - \int_{r_0}^{+\infty} \frac{d}{dr} \left\{ \left[ (-\Delta)^{-\frac{1}{2}} \phi_* \right](\exp{(r \X_j)}) \right\} \, dr \\
&= \left[ (-\Delta)^{-\frac{1}{2}} \phi_* \right](\exp{(r_0 \X_j)}) - \lim_{r \longrightarrow +\infty} \left[ (-\Delta)^{-\frac{1}{2}} \phi_* \right](\exp{(r \X_j)}).
\end{align*}

But, \eqref{equi metric d} and  the classical Li--Yau estimate for heat kernel on stratified Lie groups (c.f. for example \cite{VSC92}) imply that
\begin{align*}
(-\Delta)^{-\frac{1}{2}}(g) { \approx} \rho(g)^{1 - Q}, \quad \forall g \neq o,
\end{align*}
then we have $\left[ (-\Delta)^{-\frac{1}{2}} \phi_* \right](\exp{(r_0 \X_j)}) > 0$ and by the Fubini's theorem
\begin{align*}
\lim_{r \longrightarrow +\infty} \left[ (-\Delta)^{-\frac{1}{2}} \phi_* \right](\exp{(r \X_j)}) = 0.
\end{align*}
This leads to a contradiction. We finish the proof of this lemma.
\end{proof}

Using Lemma \ref{lem non zero}, we can deduce the following result.
\begin{proposition}\label{lem lower bound}
Fix $j=1,\ldots,n$. There exist $ 0<\varepsilon_o\ll1$ and $C(j,n)$ such that  for any $0 <\eta< \varepsilon_o$ and for all  $g \in \mathcal G$ and $r > 0$, we can find $g_* = g_*(j, g,r) \in \mathcal G$ with $\rho(g, g_*) = r$ and satisfying
\begin{align}
|K_j(g_1, g_2)| \geq C(j,n) r^{- Q}, \quad \forall g_1 \in B(g, \eta r), \  g_2 \in B(g_*, \eta r).
\end{align}
\end{proposition}
\begin{proof}
For any fixed $ j=1,\ldots,n $, by Lemma \ref{lem non zero} and the scaling property of $K_j$ (c.f. \eqref{kjs}), there exists $\tilde g_j$ satisfying
$$ \rho(\tilde g_j)=1\quad{\rm and}\quad K_j(\tilde g_j)\not=0.  $$
Since $K_j$ is a $C^\infty$ function in $\mathcal G\backslash \{o\}$, there exists $ 0<\varepsilon_o\ll1$ such that
\begin{align}\label{non zero e1}
K_j(\tilde g)\not=0\quad {\rm and} \  |K_j(\tilde g)|>{1\over 2} |K_j(\tilde g_j)|
\end{align}
for all $\tilde g\in B(\tilde g_j, 4C \varepsilon_o)$, where $C \geq 1$ is constant from \eqref{qdr}. To be more specific, we have that  for all $\tilde g\in B(\tilde g_j, 4 C \varepsilon_o)$, the values
$K_j(\tilde g)$ and $K_j(\tilde g_j)$ have the same sign.

Now for any fixed $g\in \mathcal G$, we let
\begin{align}\label{g*}
 g_*=g_*(j,g,r):= g \circ \delta_r( \tilde g_j^{-1} ).
\end{align}
Then from the choice of $g_*$, we get that
\begin{align*}
\rho(g,g_*) = \rho(g, g \circ \delta_r( \tilde g_j^{-1} )) = r \rho(\tilde g_j^{-1})= r \rho(\tilde g_j) =r.
\end{align*}

Next, for every $\eta\in(0, 2 \varepsilon_o)$, we consider the two balls $B( g,  \eta r)$ and $B( g_*,  \eta r) $.
 It is direct that for every $g_1\in B(g,  \eta r)$, we can write
 $$ g_1 = g\circ \delta_r (g'_1) $$
where $g'_1 \in B(o,  \eta)$.
Similarly, for every $g_2\in B( g_*,  \eta r)$, we can write
 $$ g_2 = g_*\circ \delta_r (g'_2) $$
 where $g'_2 \in B(o,  \eta)$.

As a consequence, we have
\begin{align}\label{dilation}
K_j(g_1,g_2) &= K_j\big(   g\circ \delta_r (g'_1) ,   g_*\circ \delta_r (g'_2)   \big)\\
&= K_j\big(   g\circ \delta_r (g'_1) ,    g \circ \delta_r( \tilde g_j^{-1} )\circ \delta_r (g'_2)   \big)\nonumber\\
&= K_j\big(    \delta_r (g'_1) ,     \delta_r( \tilde g_j^{-1} )\circ \delta_r (g'_2)   \big)\nonumber\\
&= K_j\big(    \delta_r (g'_1) ,     \delta_r( \tilde g_j^{-1} \circ g'_2)   \big)\nonumber\\
&= r^{-Q} K_j\big(    g'_1,     \tilde g_j^{-1} \circ g'_2  \big)\nonumber\\
&= r^{-Q} K_j\big(    (g'_2)^{-1} \circ    \tilde g_j \circ g'_1  \big),\nonumber
\end{align}
where the second equality comes from \eqref{g*}, the third comes from the property of the left-invariance
and the fifth comes from \eqref{kjs}.

Next, we note that
\begin{align*}
\rho\big(    (g'_2)^{-1} \circ    \tilde g_j \circ g'_1, \tilde g_j \big) &= \rho\big(      \tilde g_j \circ g'_1, g'_2 \circ  \tilde g_j \big)\\
&\leq C \, \left[ \rho\big(      \tilde g_j \circ g'_1,   \tilde g_j \big)+ \rho\big(      \tilde g_j , g'_2 \circ  \tilde g_j \big) \right]\\
&= C \, \left[  \rho\big(  g'_1,   o \big)+ \rho\big( o, g'_2 \big) \right]\\
&\leq 2 C \eta\\
&<4 C \varepsilon_o,
\end{align*}
which shows that $ (g'_2)^{-1} \circ    \tilde g_j \circ g'_1$ is contained in the ball $B(\tilde g_j, 4C \varepsilon_o)$
for all $g'_1 \in B(o,  \eta)$ and for all $g'_2 \in B(o,  \eta)$.

Thus, from  \eqref{non zero e1}, we obtain that
\begin{align}\label{lower bound e1}
| K_j\big(    (g'_2)^{-1} \circ    \tilde g_j \circ g'_1  \big)| > {1\over 2 } | K_j(\tilde g_j)|
\end{align}
and for all $g'_1 \in B(o,  \eta)$ and for all $g'_2 \in B(o,  \eta)$,
$K_j\big(    (g'_2)^{-1} \circ    \tilde g_j \circ g'_1  \big)$ and $K_j(\tilde g_j)$ have the same sign.

Now combining the equality \eqref{dilation} and \eqref{lower bound e1} above, we obtain that
\begin{align}\label{lower bound e2}
|K_j(g_1,g_2)|   > {1\over 2 } r^{-Q} |K_j(\tilde g_j)|
\end{align}
for every $g_1\in B( g,  \eta r)$ and for every $g_2\in B( g_*,  \eta r)$, where $K_j(g_1,g_2)$ and $K_j(\tilde g_j)$ have the same sign. Here $K_j(\tilde g_j)$ is a fixed
constant independent of $\eta$, $r$, $g$, $g_1$ and $g_2$. We denote
$$ C(j,n):= {1\over 2}|K_j(\tilde g_j)|.$$

From \eqref{equi metric d} and the lower bound \eqref{lower bound e2} above, we further obtain that
for every $\eta\in (0,\varepsilon_o)$,
\begin{align}\label{lower bound e3}
|K_j(g_1,g_2)|   > C(j,n) r^{-Q}
\end{align}
for every $g_1\in B( g,  \eta r)$ and for every $g_2\in B( g_*,  \eta r)$. Moreover,
the sign of $K_j(g_1,g_2)$ is invariant
for every $g_1\in B( g,  \eta r)$ and for every $g_2\in B( g_*,  \eta r)$.

This completes the proof of Lemma \ref{lem lower bound}.
\end{proof}

Next, we also introduce another version of the lower bound of the kernel of Riesz transforms.

\begin{proposition}\label{lem lower bound 2}
There exist constants $0 < \varepsilon_o \ll 1$ and $C = \min_{1 \leq j \leq n} C(j,n)$ such that for all $j= 1,\ldots,n$, $0<\eta\leq\varepsilon_o$, $g \in \mathcal G$ and $r > 0$, we can find $g_* = g_*(j,g, r) \in \mathcal G$ with $\rho(g, g_*) = r$ satisfying that
\begin{align}\label{lower bound 2 e1}
\bigg|\int_{B(g_*, \eta r)} K_j(g_1, g_2) dg_2\bigg| \geq { C \, \eta^{ Q}, \qquad \forall g_1\in B( g,  \eta r),}
\end{align}
and that
\begin{align}\label{lower bound 2 e2}
\bigg|\int_{B(g, \eta r)} K_j(g_1, g_2) dg_1\bigg| \geq { C \,} \eta^{ Q}, \qquad \forall g_2\in B( g_*,  \eta r).
\end{align}

\end{proposition}
\begin{proof}
Based on the proof of Proposition \ref{lem lower bound} above, for any fixed $g\in \mathcal G$, we let
$$
 g_*=g_*(j,g,r)
$$
as defined in \eqref{g*}. Then we have
$\rho(g,g_*) =r.$

Next, from \eqref{lower bound e3} above, we obtain that
\begin{align}\label{ee}
|K_j(g_1,g_2)|   > C(j,n) r^{-Q}
\end{align}
for every $g_1\in B( g,  \eta r)$ and for every $g_2\in B( g_*,  \eta r)$. Moreover,
the sign of $K_j(g_1,g_2)$ is invariant
for every $g_1\in B( g,  \eta r)$ and for every $g_2\in B( g_*,  \eta r)$.

As a consequence, for this particular $g_*$ and the property \eqref{ee} above, we obtain that for all $g_1\in B( g,  \eta r)$,
\begin{align*}
\bigg|\int_{B(g_*, \eta r)} K_j(g_1, g_2) dg_2\bigg| &=\int_{B(g_*, \eta r)} \big|K_j(g_1, g_2)\big| dg_2 \\
&\geq \int_{B(g_*, \eta r)} C(j,n) r^{ -Q} dg_2\\
&=C(j,n) r^{ -Q} \eta^Q r^Q\\
&=C(j,n) \eta^Q,
\end{align*}
which gives
\eqref{lower bound 2 e1}.  Symmetrically we obtain \eqref{lower bound 2 e2}.
\end{proof}

It is clear that Theorem \ref{thm0} follows directly from Proposition \ref{lem lower bound 2}.

\section{Weak factorisation of Hardy space $H^1(\mathcal G)$ and proof of Theorem \ref{thm1}} \label{S4}
\setcounter{equation}{0}

The Hardy space $H^1(\mathcal G)$ on stratified nilpotent Lie groups $\mathcal G$ is well-known and it is covered in the general framework of Folland and Stein \cite{FoSt}.  There are many equivalent ways to define the Hardy space, for example,
the Lusin and Littlewood--Paley functions \cite[Chapter 7]{FoSt}, the maximal functions  \cite[Chapters 2 and 4]{FoSt}, the atomic decomposition \cite[Chapter 3]{FoSt}. Moreover, according to the work of Christ and Geller \cite{CG}, $H^1(\mathcal G)$ also has the equivalent characterisation via the Riesz transforms $R_j = \X_j (-\Delta)^{-1/2}$. We also note that the Hardy and BMO spaces on $\mathcal G$ were introduced and studied with respect to the homogeneous
norm $\rho$, which is $C^\infty$ on $\mathcal G\backslash \{o\}$. Hence in this section we will prove Theorem \ref{thm1} using the homogeneous
norm $\rho$.

To begin with, we  point out that for $1\leq j\leq n$ and $b\in {\rm BMO}(\mathcal G)$,  $[b,R_j]$ is bounded on $L^p(\mathcal G)$ for $1<p<\infty$ with
  \begin{align} \label{upper bound}
  \|[b,R_j]: {L^p(\mathcal G)\to L^p(\mathcal G)}\| \leq C_1\|b\|_{{\rm BMO}(\mathcal G)}.
  \end{align}
In fact, \eqref{upper bound} is covered by the upper bound of commutator $[b,T]$ on a more general setting {of spaces of homogeneous type} in the sense of Coifman and Weiss. We refer to \cite[Section 2.2]{DLOWY} for the upper bound, especially see (2.14) therein.

Now it suffices to prove the lower bound of the commutator $[b,R_j]$.
We begin with recalling two equivalent definitions of $H^1(\mathcal G)$  via the atomic decomposition and via the Riesz transforms $R_j$.
\begin{definition}\label{def Hardy atom}
 The space $H^1(\mathcal G)$ is the set of functions of the form $f=\sum_{j=1}^\infty \lambda_j a_j$ with $\{\lambda_j\}\in\ell^1$ and $a_j$ an $L^\infty$-atom, meaning that it is supported on a ball $B\subset \mathcal G$, has mean value zero $\int_B a(g)dg=0$ and has a size condition $\left\Vert a\right\Vert_{L^\infty(\mathcal G)}\leq |B|^{-1}$.  One norms this space of functions by:
$$
\left\Vert f\right\Vert_{H^1(\mathcal G)}:=\inf\bigg\{\sum_{j=1}^\infty \left\vert \lambda_j\right\vert: \{\lambda_j\}\in \ell^1, f=\sum_{j=1}^\infty \lambda_j a_j, a_j \textnormal{ an atom}\bigg\}
$$
with the infimum taken over all possible representations of $f$ via atomic decompositions. Similarly, One has the definition via $L^2$-atom, meaning that the atom $a$ is supported on a ball $B\subset \mathcal G$, has mean value zero $\int_B a(g)dg=0$ and has a size condition $\left\Vert a\right\Vert_{L^2(\mathcal G)}\leq |B|^{-1/2}$.
\end{definition}

\begin{theorem}[\cite{CG}, Theorem C]\label{thm Hardy Riesz}
The Riesz transforms $R_j=\X_j (-\Delta)^{-1/2}$ are singular integral operators and $\{I, R_1, R_2,\ldots, R_n\}$ characterises $H^1(\mathcal G)$, which the equivalent norm given by:
$$
\left\Vert f\right\Vert_{H^1(\mathcal G)}:=\left\Vert f\right\Vert_{L^1(\mathcal{G})}+\sum_{j=1}^n \|R_j f\|_{L^1(\mathcal G)}.
$$
\end{theorem}

Similar to the important result about weak factorisation of the real Hardy space on Euclidean spaces obtained by Coifman, Rochberg, and Weiss \cite{CRW}, we consider the bilinear form as:
\begin{align}\label{def of pi}
\Pi_j(G,\tilde G)(g)=G(g) R_j (\tilde G)(g)-\tilde G(g)R_j^t( G)(g), \quad j=1,\ldots, n
\end{align}
for functions {$G(g)$ and ${\tilde G}(g)$} defined on $\mathcal G$,  where $R_j$ the $j$th Riesz transform and $R_j^t$ is the transpose of $R_j$.

The main result in this section is the following factorisation for $H^1(\mathcal G)$ in terms of the bilinear form $\Pi_{j}$,  which is in direct analogy with the result in the linear case obtained by Coifman, Rochberg, and Weiss in \cite{CRW}, and provides a new characterisation for the Hardy space $H^1(\mathcal G)$ set up in \cite{FoSt}.

\begin{theorem}\label{thm weakfactorization}
Suppose {$1\leq   l \leq n$}, $1<p<\infty$. Then for every function $F\in H^1(\mathcal G)$, there exist  sequences $\{\lambda_s^{(k)}\}\in \ell^1$ and  functions $G_s^{(k)}, \tilde{G}_{s}^{(k)}\in L^{\infty}_c(\mathcal G)$, the space of bounded functions with compact support, such that
\begin{align}\label{factorization}
F=\sum_{k=1}^\infty\sum_{s=1}^{\infty} \lambda_{s}^{(k)}\, { \Pi_{l}}(G_s^{(k)}, \tilde {G}_{s}^{(k)})
\end{align}
in the sense of $H^1(\mathcal G)$.
Moreover, we have that:
$$
\left\Vert F\right\Vert_{H^1(\mathcal G)} \approx \inf\left\{\sum_{k=1}^\infty \sum_{s=1}^{\infty} \left\vert \lambda_s^{(k)}\right\vert \left\Vert G_s^{(k)}\right\Vert_{L^{p}(\mathcal G)}\left\Vert \tilde{G}_{s}^{(k)}\right\Vert_{L^{p'}(\mathcal G)}  \right\},
$$
where the infimum  is taken over all possible representations of $F$ that satisfy \eqref{factorization} and $p'$ is the conjugate of $p$, i.e., $\frac{1}{p}+\frac{1}{p'}=1$.
\end{theorem}

Now we consider the boundedness property of the bilinear form $\Pi_j(G,\tilde G)(g)$.
\begin{proposition}\label{prop H1 estimate of pi}
 For any fixed $\tilde G, G\in L^{\infty}_c(\mathcal G)$, we have that $\Pi_{j}(G,\tilde G)$ is in $H^1(\mathcal G)$. Moreover,
there exists a positive constant {$C$,  independent of $(\tilde G, G)$, }such that
\begin{align}\label{H1norm}
\|\Pi_{j}(G,\tilde G)\|_{H^1(\mathcal G)}\leq C\|\tilde G\|_{L^{p}(\mathcal G)}\|G\|_{L^{p'}(\mathcal G)}.
\end{align}
\end{proposition}

\begin{proof}
For any fixed $\tilde G, G\in L^{\infty}_c(\mathcal G)$,
to show $\Pi_{j}(G,\tilde G)$ is in $H^1(\mathcal G)$ with the required norm \eqref{H1norm}, we now consider the properties of size estimate, cancellation and compact support for
$\Pi_{j}(G,\tilde G)$.

To begin with,
since $\tilde G, G\in L^{\infty}_c(\mathcal G)$,
we have that $\tilde G\in L^{p'}(\mathcal G)$ and $G\in L^{p}(\mathcal G)$ for any $p,p'\in (1, \infty)$ with $1 ={1\over p} +{1\over p'}$.
Then, from the definition of $\Pi_{j}$ as in \eqref{def of pi}, the {$L^{p'}$-}boundedness of the Riesz transform $R_j$ and H\"older's inequality, we have that $\Pi_{j}(G,\tilde G) \in L^1(\mathcal G)$ with the estimate $\left\Vert \Pi_j(G,\tilde{G})\right\Vert_{L^1(\mathcal{G})}\leq C \|\tilde G\|_{L^{p}(\mathcal G)}\|G\|_{L^{p'}(\mathcal G)}$.

Moreover, note that from the definition of $\Pi_{l}$ as in \eqref{def of pi}, we have
$$\int_{\mathcal G} \Pi_{j}(G,\tilde G)(g)\,dg=0.$$
Next,  since $\tilde G, G\in L^{\infty}_c(\mathcal G)$, from the definition of $\Pi_{j}$ as in \eqref{def of pi} and the {$L^2$-}boundedness of $R_j$, it is direct to see that
$\Pi_{j}(G,\tilde G)$ is in $L^2(\mathcal G)$ with compact support. Hence, we immediately have that
$\Pi_{j}(G,\tilde G)$ is a multiple of an $H^1(\mathcal G)$ atom, i.e, we get that
$\Pi_{j}(G,\tilde G)$ is in $H^1(\mathcal G)$.  Then it suffices to verify that the $H^1(\mathcal G)$ norm of $\Pi_{j}(G,\tilde G)$
satisfies \eqref{H1norm}.

To see this, for $b\in {\rm BMO}(\mathcal G)$, we now consider the inner product
\begin{align}\label{inner}
\langle b, \Pi_{j}(G,\tilde G)\rangle:=\int_{\mathcal G} b(g)\Pi_{j}(G,\tilde G)(g)\,dg.
\end{align}
It is clear that this inner product is well-defined since $\Pi_{j}(G,\tilde G)(g)$ is in $L^2(\mathcal G)$ with compact support and  $b\in {\rm BMO}(\mathcal G)$ and hence in $L^2_{loc}(\mathcal G)$.

We now  claim that for any fixed $\tilde G, G\in L^{\infty}_c(\mathcal G)$,
\begin{align}\label{exchange}
\langle b, \Pi_{j}(G,\tilde G)\rangle= \langle G, [b, R_j](\tilde G)\rangle
\end{align}
In fact, since $\tilde G, G$ are in $L^{\infty}_c(\mathcal G)$ and
$b$ is in $L^2_{loc}(\mathcal G)$, we have
$\langle b,  G\, R_j(\tilde G)\rangle
=\langle G,  b\, R_j(\tilde G)\rangle.
$
Moreover,
 we also have
$\langle b,  \tilde  G\, R_{j}^t(G)\rangle
=\langle G, R_j(b\tilde  G)\rangle.
$
Combining these two equalities and the definition of $\Pi_{j}$ as in \eqref{def of pi}, we have
\begin{align*}
\langle b, \Pi_{j}(G, \tilde G)\rangle
&=\langle b,   G\, R_j (\tilde G)-\tilde G\, R_j^t( G)\rangle= \langle G,\, b\, R_j(\tilde G)-   R_j(b\tilde  G)\rangle \\
&= \langle G, [b, R_j](\tilde G)\rangle,
\end{align*}
which implies \eqref{exchange}.

Now from the equality \eqref{exchange} and the upper bound of the commutator as in \eqref{upper bound}, we obtain that
\begin{align}\label{PiNorm}
\big|\langle b, \Pi_{j}(G,\tilde  G)\rangle\big|= \big|\langle G, [b, R_j](\tilde G)\rangle\big|\leq C \|b\|_{ {\rm BMO}(\mathcal G)}  \|\tilde G\|_{L^{p}(\mathcal G)}\|G\|_{L^{p'}(\mathcal G)}.
\end{align}

We then verify  \eqref{H1norm}.
To see this, we point out that from the fundamental fact as in \cite[Exercise 1.4.12 (b)]{Gra}, we have
$$ \|\Pi_{j}(G, \tilde G)\|_{ H^1(\mathcal G)} \approx \sup_{b:\ \|b\|_{ {\rm BMO}(\mathcal G)}\leq1  }
\big| \langle  b, \Pi_{j}(G, \tilde G)\rangle \big|, $$
which, together with \eqref{PiNorm}, immediately implies that
\eqref{H1norm} holds.  The proof of Proposition \ref{prop H1 estimate of pi} is complete.
\end{proof}

Next, we need the following  technical lemma about certain $H^1(\mathcal G)$ functions.
\begin{lemma}\label{lemma Hardy}
Let $\varepsilon_0$ be the constant as in Lemma \ref{lem lower bound 2} satisfying $0 < \varepsilon_o \ll 1$, and let $r>0$, $0<\eta\leq \varepsilon_o$.
Suppose $F$ is a function defined on $\mathcal G$ satisfying:  $\int_{\mathcal G} F(\tilde{g})\,d\tilde{g}=0$, and $|F(\tilde{g})|\leq \chi_{B(g,\eta r)}(\tilde{g})+\chi_{B(g_*,\eta r)}(\tilde{g})$, where {$\rho(g,g_*)=r$.} Then we have
\begin{align}\label{lemma Hardy e1}
{\| F \|_{H^1(\mathcal G)}} {\lesssim} \eta^Q r^Q\log\frac{1}{\eta} .
\end{align}
\end{lemma}
\begin{proof}

Let $F$ satisfy $\int_{\mathcal G} F(\tilde{g})\,d\tilde{g}=0$, and $|F(\tilde{g})|\leq \chi_{B(g,\eta r)}(\tilde{g})+\chi_{B(g_*,\eta r)}(\tilde{g})$ with {$\rho(g,g_*)=r$}. Obviously we have $\|F\|_{L^1(\mathcal G)} {\lesssim} \eta^Q r^Q$.
To show \eqref{lemma Hardy e1}, by Theorem \ref{thm Hardy Riesz}, we will need to estimate the
$L^1(\mathcal G)$ norm of $F$ and $R_j F$, $j=1,\ldots,n$.

Set in the sequel,
$c_* = 4 \, \left( C + C_2 \right),$
where $C, C_2 \geq 1$ come from \eqref{qdr} and \eqref{equi metric d} respectively. Without loss of generality, we may assume that $\varepsilon_0$ small enough such that $4 \, c_*^2 \, \varepsilon_0 \leq 1$.

Next, for each Riesz transform $R_j$, we have
\begin{align*}
\|R_j F\|_{L^1(\mathcal G)} &= \int_{ B(g, c_* \eta r) }|R_j F(\tilde g)| \,d\tilde g+\int_{B(g_*, c_* \eta r)}|R_j F(\tilde g)| \,d\tilde g \\
&\quad+ \int_{\mathcal G\backslash\big( B(g, c_* \eta r) \cup B(g_*, c_* \eta r)\big)}|R_j F(\tilde g)| \,d\tilde g\\
&= I+II+III.
\end{align*}
For the terms $I$ and $II$, by using the $L^2(\mathcal G)$ boundedness of the Riesz  transform $R_j$, we get that
\begin{align*}
I+II \lesssim  |B(g, c_* \eta r)|^{1\over 2}  \|F\|_{L^2(\mathcal G)} +|B(g_*, c_* \eta r)|^{1\over 2}  \|F\|_{L^2(\mathcal G)} \lesssim \eta^Q r^Q.
\end{align*}
For the term $III$, from the cancellation condition of $F$, we get that
\begin{align*}
III &= \int_{\mathcal G\backslash\big( B(g, c_* \eta r) \cup B(g_*, c_* \eta r)\big)} \bigg| \int_{\mathcal G}\Big[K_j(\tilde g,\tilde{\tilde g})-K_j(\tilde g,g)\Big] F(\tilde{\tilde g})\, d\tilde{\tilde g}\bigg| \,d\tilde g\\
&\leq \int_{\mathcal G\backslash\big( B(g, c_* \eta r) \cup B(g_*, c_* \eta r)\big)}  \int_{B(g,\eta r)}\Big|K_j(\tilde g,\tilde{\tilde g})-K_j(\tilde g,g)\Big| \, d\tilde{\tilde g} \,d\tilde g\\
&\quad+ \int_{\mathcal G\backslash\big( B(g, c_* \eta r) \cup B(g_*, c_* \eta r)\big)}  \int_{B(g_*,\eta r)}\Big|K_j(\tilde g,\tilde{\tilde g})-K_j(\tilde g,g)\Big| \, d\tilde{\tilde g} \,d\tilde g\\
&=: III_1 +III_2.
\end{align*}
For the term $III_1$, we further have
that
\begin{align*}
III_1
&\leq \int_{\big( B(g, 4 C_2 r) \big)^c}  \int_{B(g,\eta r)}\Big|K_j(\tilde g,\tilde{\tilde g})-K_j(\tilde g,g)\Big| \, d\tilde{\tilde g} \,d\tilde g\\
&\quad+ \int_{B(g, 4 C_2 r) \backslash B(g, c_* \eta r)}  \int_{B(g,\eta r)}\Big|K_j(\tilde g,\tilde{\tilde g})-K_j(\tilde g,g)\Big| \, d\tilde{\tilde g} \,d\tilde g\\
&=: III_{11}+III_{12},
\end{align*}
where the constant $C_2 \geq 1$ comes from \eqref{equi metric d}.

Using \eqref{MEHT} and \eqref{equi metric d}, it is direct that
\begin{align*}
III_{11}
&\lesssim \int_{\big( B(g,4 C_2  r) \big)^c}  \int_{B(g,\eta r)} {d(g,\tilde{\tilde g}) \over d(g,\tilde g)^{Q+1}} \, d\tilde{\tilde g} \,d\tilde g \lesssim \int_{\big( B(g,4 C_2 r) \big)^c} {(\eta r)^{Q + 1} \over \rho(g,\tilde g)^{Q+1}} \,d\tilde g\\
& \lesssim \eta^{Q + 1} r^Q,
\end{align*}
where the last inequality follows from the decomposition of the set  $\big( B(g, 4 C_2 r) \big)^c$ into annuli $\cup_{j\geq 2} B(g, 2^{j+1} C_2 r) \backslash B(g, 2^j C_2 r)$.

For $III_{12}$, by \eqref{MEHT} and \eqref{equi metric d}, we have
\begin{align*}
III_{12} &\leq \int_{B(g,4 C_2  r) \backslash  B(g, c_* \eta r) }  \int_{B(g,\eta r)}\Big(\Big|K_j(\tilde g,\tilde{\tilde g})\Big|+\Big|K_j(\tilde g,g)\Big|\Big) \, d\tilde{\tilde g} \,d\tilde g\\
&\lesssim \int_{B(g,4 C_2  r) \backslash  B(g, c_* \eta r) }  \int_{B(g,\eta r)}\Big( {1\over \rho(\tilde g,\tilde{\tilde g})^Q} +{1\over \rho(\tilde g, g)^Q} \Big) \, d\tilde{\tilde g} \,d\tilde g\\
&\lesssim \int_{B(g,4 C_2  r) \backslash  B(g, c_* \eta r) }  \int_{B(g,\eta r)} {1\over \rho(g, \tilde g)^Q} \, d\tilde{\tilde g} \,d\tilde g\\
&\lesssim \eta^Q r^Q\log{1\over \eta}.
\end{align*}

Symmetrically we can obtain the estimate for $III_2$, which is also bounded by $C\eta^Q r^Q + C\eta^Q r^Q\log{1\over \eta}$.
Combining all the estimates above, and by noting that $1\ll \log{1\over \eta} $, we obtain that
$${\| F \|_{H^1(\mathcal G)}} \leq C \eta^Q r^Q\log\frac{1}{\eta} ,
$$
which implies that \eqref{lemma Hardy e1} holds.
\end{proof}

Suppose $1\leq j\leq n$.  Ideally, given an $H^1(\mathcal G)$-atom $a$ (as in Definition \ref{def Hardy atom}), we would like to find functions $\tilde{G}\in L^{p'}(\mathcal G)$, $G\in L^{p}(\mathcal G)$ such that $\Pi_j(\tilde G, G)=a$ pointwise.  While we are not able to do this in general, our construction in  the next proposition is close to that aim.

\begin{proposition}
\label{prop:ApproxFactorization}
Suppose $a(g)$ is an arbitrary  $H^1(\mathcal G)$-atom as in Definition \ref{def Hardy atom}.
Suppose $1\leq j\leq n$ and $1<p,p'<\infty$  with
$ {1\over p} +  {1\over p'} = 1$.  For every $\varepsilon>0$,
there exist $\tilde G, G\in L^{\infty}_c(\mathcal G)$
and a positive number $\eta$ depending only on $\varepsilon$
such that
$$
\left\Vert a-\Pi_{j}(G,\tilde G)\right\Vert_{H^1(\mathcal G)}<\varepsilon
$$
and that $\big\Vert \tilde G\big\Vert_{L^{p}(\mathcal G)}\left\Vert G\right\Vert_{L^{p'}(\mathcal G)}\leq C \eta^{-Q}$, where $C$ is an absolute positive constant.
\end{proposition}

\begin{proof}
Let $a(\tilde g)$ be an $H^1(\mathcal G)$-atom, supported in $B(g,\bar r)\subset \mathcal G$, satisfying that
$$ \int_{\mathcal G}a(\tilde g)d\tilde g=0 \quad {\rm and} \quad \|a\|_{L^\infty(\mathcal G)} \leq {|B(g,\bar r)|^{-1} = |B(o,\bar r)|^{-1} = \bar r^{-Q}.}$$

Note that from {Proposition} \ref{lem lower bound 2},
there exist constants $0 < \varepsilon_o \ll 1$  and {$C > 0$}, such that for all $0<\eta\leq\varepsilon_o$, $1 \leq j \leq n$, $g \in \mathcal G$ and $r > 0$, we can find $g_* = g_*(j, r) \in \mathcal G$ satisfying $\rho(g, g_*) = r$,
\begin{align}\label{e1 proof thm1}
\bigg|\int_{B(g_*, \eta r)} K_j(g_1, g_2) dg_2\bigg| \geq {C \eta^{ Q}, \quad \forall g_1 \in B(g, \eta r).}
\end{align}

Now we fix $g$, which is the centre of the ball $B(g,\bar r)$ above.
Next, fix $1\leq j\leq n$ and fix $\varepsilon>0$.  Choose $\eta$  sufficiently small so that $0<\eta<\varepsilon_o$
and that
 $$  \eta \log \frac{1}{\eta} <\varepsilon.$$

Then we set $r=\bar r \cdot \eta^{-1}$. Now based on the argument above, there exits $g_*\in \mathcal G$ satisfying {$\rho(g, g_*) = r$} and
\eqref{e1 proof thm1}.

We now define $\tilde G(\tilde g)=\chi_{ B(g_*, \eta r) }(\tilde g)$ and $ G(\tilde g) =  {\displaystyle a(\tilde g)\over R_j(\tilde G)( g)}$, $\tilde g\in \mathcal G$.

From the definitions of the functions $G$ and $\tilde G$, we obtain that {$\operatorname{Supp}\,G \subseteq \overline {B(g,\eta r)}$ and $\operatorname{Supp}\,\tilde G=\overline{B(g_*,\eta r)}$}. Moreover, it is direct to see that
$\|\tilde G\|_{L^{p}(\mathcal G)} \approx (\eta r)^{\frac{Q}{p}} $ and that
$$\|G\|_{L^{p'}(\mathcal G)} =  \frac{1}{|R_j(\tilde G)( g)|} \|a\|_{L^{p'}(\mathcal G)} {\lesssim} \eta^{-Q} (\eta r)^{-Q} (\eta r)^{{Q\over p'}}. $$
Hence, we obtain that
\begin{align*}
\|\tilde G\|_{L^{p}(\mathcal G)}\|G\|_{L^{p'}(\mathcal G)}  &\leq C (\eta r)^{\frac{Q}{p}}\  \eta^{-Q} (\eta r)^{-Q} (\eta r)^{{Q\over p'}}
=  C\eta^{-Q} .
\end{align*}
Next, we have
\begin{align*}
a(\tilde g)-\Pi_{j}( G, \tilde G)(\tilde g)
&=a(\tilde g)- \Big( G(\tilde g) R_j(\tilde  G)(\tilde g) - \tilde G(\tilde g) R_j^t( G)(\tilde g) \Big)\\
&=\bigg( a(\tilde g) -   {\displaystyle a(\tilde g)\over R_j(\tilde G)( g)}\, R_j(\tilde G)(\tilde g) \bigg) +{ \tilde G(\tilde g) R_j^t(a)(\tilde g) \over R_j(\tilde G)(g) } \\
&=: W_1(\tilde g)+W_2(\tilde g).
\end{align*}

By definition, it is obvious that $W_1(\tilde g)$ is supported on $B(g,\eta r)$ and $W_2(\tilde g)$ is supported on $B(g_*,\eta r)$.  We first estimate $W_1$.  For $\tilde g\in B(g,\eta r)$, {it follow from \eqref{e1 proof thm1}, \eqref{MEHT} and \eqref{equi metric d} that}
\begin{align*}
|W_1(\tilde g)|
 &= |a(\tilde g)|\frac{|R_j(\tilde G)( g)-R_j(\tilde G)(\tilde g)|}
{|R_j(\tilde G)( g)|}\\
&\lesssim  { \|a\|_{L^\infty(\mathcal G)} \over \eta^{Q}} \int_{B(g_*,\eta r)} |K_j(g,g_2)- K_j(\tilde g,g_2)  | \,dg_2\\
&\lesssim  { \|a\|_{L^\infty(\mathcal G)} \over \eta^{Q}}  \int_{B(g_*,\eta r)} {{\rho(g,\tilde g) \over \rho(g,g_*)^{Q+1}}} \ dg_2\\
&\lesssim  { (\eta r)^{-Q} \over \eta^{Q}}   (\eta r)^{Q} {\eta r\over r^{Q+1}}\\
&\lesssim \frac{\eta}{ \bar r^{Q}}.
\end{align*}
Hence we obtain that $$ |W_1(\tilde g)|\lesssim\frac{\eta}{ \bar r^{Q}} \chi_{B(g,\bar r)}(\tilde g).$$
Next we estimate $W_2(\tilde g)$. From the definition of $G(\tilde g)$ and $\tilde G(\tilde g)$, we have
\begin{align*}
|W_2(\tilde g)|
 &= { \chi_{ B(g_*, \eta r) }(\tilde g) |R_j^t(a)(\tilde g)| \over |R_j(\tilde G)(g)| }\\
 &= { \chi_{ B(g_*, \eta r) }(\tilde g)  \over |R_j(\tilde G)(g)| } \bigg| \int_{B(g,\eta r)}\Big[ K_j^t(\tilde g, g_2)- K_j^t(\tilde g, g)\Big] a(g_2) dg_2 \bigg|\\
 &\lesssim {  \ \|a\|_{L^\infty(\mathcal G)} \over \eta^{Q} }  \int_{B(g,\eta r)}  {d(g, g_2) \over d(g,g_*)^{Q+1}} dg_2 \\
&\lesssim { (\eta r)^{-Q} \over \eta^{Q}}   (\eta r)^{Q} {\eta r\over r^{Q+1}}\\
&\lesssim\frac{\eta}{ \bar r^{Q}},
\end{align*}
where we use the cancellation property of the atom $a(g_2)$ in the second equality,  and the fact that $K_j^t(g_1, g_2) = {K_j}(g_2, g_1)$ {and \eqref{MEHT}} in the first inequality, {also \eqref{equi metric d} in the second inequality.}
Hence we have $$ |W_2(\tilde g)|\lesssim\frac{\eta}{ \bar r^{Q}} \chi_{B(g_*,\bar r)}(\tilde g).$$

Combining the estimates of $W_1$ and $W_2$, we obtain that
\begin{align}\label{size}
 \Big|a(\tilde g)-\Pi_{j}(G, \tilde  G)(\tilde g)\Big|\lesssim  \frac{\eta}{ \bar r^{Q}} (\chi_{B(g_*,\bar r)}(\tilde g)+\chi_{B(g,\bar r)}(\tilde g)).
\end{align}
Next we point out that
\begin{align}\label{cancellation R}
\int_{\mathcal G} \Big[a(\tilde g)-\Pi_{j}(G, \tilde  G)(\tilde g)\Big) \Big] d\tilde g=0
\end{align}
since the atom $a(\tilde g)$ has cancellation and the second integral equals 0 just by the definitions of $\Pi_{j}(G, \tilde  G)(\tilde g)$.

Then the size estimate \eqref{size} and the cancellation \eqref{cancellation R},  together with \eqref{lemma Hardy e1} in Lemma \ref{lemma Hardy}, imply that
$$ \big\|a - \Pi_{j}(G, \tilde G) \big\|_{H^1(\mathcal G)} \leq  C \eta\log\frac{1}{\eta} <C\varepsilon. $$
This proves the proposition.
\end{proof}

We now prove Theorem \ref{thm weakfactorization}.
\begin{proof}[Proof of Theorem \ref{thm weakfactorization}]
By {Proposition \ref{prop H1 estimate of pi}}, we have that
$$
\left\Vert \Pi_l(G,\tilde G)\right\Vert_{H^1(\mathcal G)}\leq C \left\Vert G\right\Vert_{L^{p'}(\mathcal G)} \left\Vert \tilde G\right\Vert_{L^{p}(\mathcal G)},
$$ it is immediate that we have for any representation of $$F=\sum_{k=1}^\infty\sum_{j=1}^\infty\lambda_j^{(k)}\, \Pi_l(G_j^{(k)},\tilde G_{j}^{(k)})$$ that
\begin{align*}
{\left\Vert F \right\Vert_{H^1(\mathcal G)}}& \leq C \inf\left\{\sum_{k=1}^\infty\sum_{j=1}^{\infty}\left\vert \lambda_j^{(k)}\right\vert \left\Vert G_j^{(k)}\right\Vert_{L^{p'}(\mathcal G)}\left\Vert \tilde G_j^{(k)}\right\Vert_{L^p(\mathcal G)}: \right.\\
 &\hskip4cm\left. {F = } \sum_{k=1}^\infty\sum_{j=1}^{\infty}\lambda_j^{(k)} \, \Pi_l(G_j^{(k)},\tilde G_{j}^{(k)}) \right\}.
\end{align*}

We turn to show that the other inequality hold and that it is possible to obtain such a decomposition for any $F \in H^1(\mathcal G)$.  By the atomic decomposition for $H^1(\mathcal G)$, for any $F \in H^1(\mathcal G)$ we can find a sequence $\{\lambda_{j}^{(1)}\}\in \ell^1$ and sequence of $H^1(\mathcal G)$-atoms $a_j^{(1)}$ so that $F =\sum_{j=1}^{\infty} \lambda_j^{(1)} a_{j}^{(1)}$ and $\sum_{j=1}^{\infty} \left\vert \lambda_j^{(1)}\right\vert \leq C_0 \left\Vert F \right\Vert_{H^1(\mathcal G)}$.

We explicitly track the implied absolute constant $C_0$ appearing from the atomic decomposition since it will play a role in the convergence of the approach.  Fix $\varepsilon>0$ so that $\varepsilon C_0<1$. Then we also have small positive number $\eta\ll 1$ with $\eta \log\frac{1}{\eta} < \epsilon$.  We apply Proposition \ref{prop:ApproxFactorization} to each atom $a_{j}^{(1)}$.  So there exists $G_j^{(1)}, \tilde G_j^{(1)}\in L_c^\infty(\mathcal G)$ with compact supports and satisfying
$\left\Vert G_j^{(1)}\right\Vert_{L^{p'}(\mathcal G)}\left\Vert \tilde G_j^{(1)}\right\Vert_{L^{p}(\mathcal G)}\leq C\eta^{-Q} $ and
\begin{equation*}
\left\Vert a_j^{(1)}-\Pi_l(G_j^{(1)}, \tilde G_j^{(1)})\right\Vert_{H^1(\mathcal G)}<\varepsilon,\quad \forall j.
\end{equation*}
Now note that we have
\begin{align*}
F  &=  \sum_{j=1}^{\infty} \lambda_{j}^{(1)} a_j^{(1)}=   \sum_{j=1}^{\infty} \lambda_{j}^{(1)} \,\Pi_l(G_j^{(1)}, \tilde G_j^{(1)})+\sum_{j=1}^{\infty} \lambda_{j}^{(1)} \left(a_j^{(1)}-\Pi_l(G_j^{(1)}, \tilde G_j^{(1)})\right) \\
&:=  M_1+E_1.
\end{align*}
Observe that we have
\begin{align*}
\left\Vert E_1\right\Vert_{H^1(\mathcal G)}  &\leq  \sum_{j=1}^{\infty} \left\vert \lambda_j^{(1)}\right\vert \left\Vert a_j^{(1)}-\Pi_l(G_j^{(1)}, \tilde G_j^{(1)})\right\Vert_{H^1(\mathcal G)} \leq  \varepsilon \sum_{j=1}^{\infty} \left\vert \lambda_j^{(1)}\right\vert\\
& \leq \varepsilon C_0\left\Vert F \right\Vert_{H^1(\mathcal G)}.
\end{align*}
We now iterate the construction on the function $E_1$.  Since $E_1\in H^1(\mathcal G)$, we can apply the atomic decomposition in {$H^1(\mathcal G)$ to} find a sequence $\{\lambda_j^{(2)}\}\in \ell^1$ and a sequence of $H^1(\mathcal G)$-atoms $\{a_j^{(2)}\}$ so that $E_1=\sum_{j=1}^{\infty} \lambda_j^{(2)} a_{j}^{(2)}$ and
$$
\sum_{j=1}^{\infty} \left\vert \lambda_j^{(2)}\right\vert \leq C_0 \left\Vert E_1\right\Vert_{H^1(\mathcal G)}\leq \varepsilon C_0^2 \left\Vert F \right\Vert_{H^1(\mathcal G)}.
$$

Again, we will apply Proposition \ref{prop:ApproxFactorization} to each atom $a_{j}^{(2)}$.  So there exist $\tilde G_j^{(2)}, G_j^{(2)}\in L_c^\infty(\mathcal G)$ with compact supports and satisfying  the condition that $\left\Vert
G_j^{(2)}\right\Vert_{L^{p'}(\mathcal G)}\left\Vert \tilde G_j^{(2)}\right\Vert_{L^{p}(\mathcal G)}\leq C \eta^{-Q}$ and
\begin{equation*}
\left\Vert a_j^{(2)}-\Pi_l(G_j^{(2)},\tilde  G_j^{(2)})\right\Vert_{H^1(\mathcal G)}<\varepsilon,\quad \forall j.
\end{equation*}
We then have that:
\begin{align*}
E_1 & =  \sum_{j=1}^{\infty} \lambda_{j}^{(2)} a_j^{(2)}=   \sum_{j=1}^{\infty} \lambda_{j}^{(2)}\, \Pi_l(G_j^{(2)},\tilde  G_j^{(2)})+\sum_{j=1}^{\infty} \lambda_{j}^{(2)} \left(a_j^{(2)}-\Pi_l(G_j^{(2)},\tilde  G_j^{(2)})\right) \\
&:=  M_2+E_2.
\end{align*}
But, as before observe that
\begin{align*}
\left\Vert E_2\right\Vert_{H^1(\mathcal G)}&  \leq  \sum_{j=1}^{\infty} \left\vert \lambda_j^{(2)}\right\vert \left\Vert a_j^{(2)}-\Pi_l(G_j^{(2)},\tilde  G_j^{(2)})\right\Vert_{H^1(\mathcal G)} \leq  \varepsilon \sum_{j=1}^{\infty} \left\vert \lambda_j^{(2)}\right\vert\\
& \leq \left(\varepsilon C_0\right)^{2}\left\Vert F \right\Vert_{H^1(\mathcal G)}.
\end{align*}
And, this implies for $F$ that we have:
\begin{eqnarray*}
F & = & \sum_{j=1}^{\infty} \lambda_{j}^{(1)} a_j^{(1)} =   \sum_{j=1}^{\infty} \lambda_{j}^{(1)} \,\Pi_l(G_j^{(1)}, \tilde G_j^{(1)})+\sum_{j=1}^{\infty} \lambda_{j}^{(1)} \left(a_j^{(1)}-\Pi_l(G_j^{(1)},\tilde  G_j^{(1)})\right)\\
 & = & M_1+E_1=M_1+M_2+E_2 \\
 &=&  \sum_{k=1}^{2} \sum_{j=1}^{\infty} \lambda_{j}^{(k)} \,\Pi_l(G_j^{(k)},\tilde  G_j^{(k)})+E_2.
\end{eqnarray*}

Repeating this construction for each $1\leq k\leq K$ produces functions $\tilde G_j^{(k)}, G_j^{(k)}\in L_c^\infty(\mathcal G)$ with compact supports and  satisfying the condition that $\left\Vert \tilde G_j^{(k)}\right\Vert_{L^{p}(\mathcal G)}\left\Vert G_j^{(k)}\right\Vert_{L^{p'}(\mathcal G)}\leq C \eta^{-Q} $ for all $j$, sequences $\{\lambda_{j}^{(k)}\}\in \ell^1$ with $\left\Vert \{\lambda_{j}^{(k)}\}\right\Vert_{\ell^1}\leq \varepsilon^{k-1} C_0^k \left\Vert F \right\Vert_{H^1(\mathcal G)}$, and a function $E_K\in H^1(\mathcal G)$ with $\left\Vert E_K\right\Vert_{H^1(\mathcal G)}$ $\leq \left(\varepsilon C_0\right)^{K}\left\Vert F \right\Vert_{H^1(\mathcal G)}$ so that
$$
F =\sum_{k=1}^{K} \sum_{j=1}^\infty \lambda_{j}^{(k)} \,\Pi_l(G_j^{(k)},\tilde  G_j^{(k)})+E_K.
$$
Passing $K\to\infty$ gives the desired decomposition:  $$F =\sum_{k=1}^{\infty} \sum_{j=1}^\infty \lambda_{j}^{(k)}\, \Pi_l(G_j^{(k)},\tilde  G_j^{(k)})$$ in the sense of $H^1(\mathcal G)$.  We also have that:
$$
\sum_{k=1}^{\infty} \sum_{j=1}^\infty\left\vert \lambda_{j}^{(k)}\right\vert \leq \sum_{k=1}^{\infty} \varepsilon^{-1} (\varepsilon C_0)^{k} \left\Vert F \right\Vert_{H^1(\mathcal G)}= \frac{ C_0}{1-\varepsilon C_0}\left\Vert F \right\Vert_{H^1(\mathcal G)}.
$$
The proof of Theorem \ref{thm weakfactorization} is complete.
\end{proof}

We now turn to the proof of Theorem \ref{thm1}.

\begin{proof}[Proof of Theorem \ref{thm1}]

The upper bound in this theorem is pointed out in \eqref{upper bound}.  It suffices to consider only the  lower bound.

Suppose now $b\in {\rm BMO}(\mathcal G)$ such that $[b,R_l]$ is bounded on $L^{p}(\mathcal{G})$ for some $1<p<\infty$.
We now show that
$$ \left\Vert b\right\Vert_{ {\rm BMO}(\mathcal{G})}
\leq C\|[b,R_l] : L^{p}(\mathcal{G})\to L^{p}(\mathcal{G})\|. $$

From Theorem \ref{thm weakfactorization}, we already know that for every $f$ in $H^1(\mathcal G)$, we have a constructive weak factorisation
for $f$, i.e, for $f\in H^1(\mathcal{G})$,
there exists a weak  factorisation of $f$ such that
\begin{align}\label{wf}
 f=\sum_{k=1}^{\infty} \sum_{s=1}^\infty \lambda_{s}^{(k)} \,\Pi_{l}(G_s^{(k)},\tilde  G_{s}^{(k)})
\end{align}
in the sense of $H^1(\mathcal{G})$, where the sequences $\{\lambda_s^{(k)}\}\in \ell^1$ and  functions $ G_s^{(k)}, \tilde G_{s}^{(k)}\in L^{\infty}_c(\mathcal{G})$,
and that
$$
\sum_{k=1}^{\infty} \sum_{s=1}^\infty\left\vert \lambda_{s}^{(k)}\right\vert \left\Vert \tilde G_s^{(k)}\right\Vert_{L^{p}(\mathcal{G})}  \left\Vert G_{s}^{(k)}\right\Vert_{L^{p'}(\mathcal{G})} \leq  C\left\Vert f\right\Vert_{H^1(\mathcal{G})}.
$$

As a consequence, we obtain that
\begin{align}\label{eeee}
\langle b,f \rangle &=    \left\langle b, \sum_{k=1}^{\infty} \sum_{s=1}^\infty \lambda_{s}^{(k)} \,\Pi_{l}(G_s^{(k)},\tilde  G_{s}^{(k)}) \right\rangle \nonumber\\
&=\sum_{k=1}^{\infty} \sum_{s=1}^\infty \lambda_{s}^{(k)} \, \left\langle b, \Pi_{l}(G_s^{(k)},\tilde  G_{s}^{(k)}) \right\rangle\nonumber\\
&=\sum_{k=1}^{\infty} \sum_{s=1}^\infty \lambda_{s}^{(k)} \, \left \langle G_s^{(k)} , [b,R_l](\tilde G_{s}^{(k)})\right\rangle,\nonumber
\end{align}
which the last equality follows from \eqref{exchange}. This yields that
$$
|\langle b,f \rangle| \leq \sum_{k=1}^{\infty} \sum_{s=1}^\infty |\lambda_{s}^{(k)}| \  |\langle G_s^{(k)} , [b,R_l](\tilde G_{s}^{(k)})\rangle|,
$$
which is further controlled by
\begin{align*}
& \sum_{k=1}^{\infty}\sum_{s=1}^\infty \left\vert \lambda_s^{(k)}\right\vert \left\Vert [b,R_l](\tilde G_{s}^{(k)})\right\Vert_{L^p(\mathcal{G})}\left\Vert G_s^{(k)}\right\Vert_{L^{p'}(\mathcal{G})}\\
& \leq  \|[b,R_l] : L^{p}(\mathcal{G}) \to L^{p}(\mathcal{G})\| \ \ \sum_{k=1}^{\infty}\sum_{s=1}^\infty \left\vert \lambda_s^{(k)}\right\vert \left\Vert \tilde G_s^{(k)}\right\Vert_{L^{p}(\mathcal{G})} \left\Vert G_{s}^{(k)}\right\Vert_{L^{p'}(\mathcal{G})}\\
& \leq   C \|[b,R_l] : L^{p}(\mathcal{G}) \to L^{p}(\mathcal{G})\| \left\Vert f\right\Vert_{H^1(\mathcal{G})}.
\end{align*}
By the duality between $ {\rm BMO}(\mathcal{G})$ and $H^1(\mathcal{G})$, we have that:
\begin{align*}
\left\Vert b\right\Vert_{ {\rm BMO}(\mathcal{G})}&\approx \sup_{f\in H^1(\mathcal{G}):\ \left\Vert f\right\Vert_{H^1(\mathcal{G})}\leq 1} \left\vert \int_{\mathcal{G}} b(g)f(g)\,dg\right\vert\\
&\leq C\|[b,R_l] : L^{p}(\mathcal{G})\to L^{p}(\mathcal{G})\|.
\end{align*}
The proof of Theorem \ref{thm1} is complete.
\end{proof}

\section{Application: div-curl Lemma and Proof of Theorem \ref{thm2} }
\setcounter{equation}{0}

In this section, we introduce a version of the curl operator and establish the div-curl lemma on stratified nilpotent Lie groups $\mathcal G$ with respect to Hardy space $H^1(\mathcal G)$, based on the boundedness of the commutator in Theorem \ref{thm1}.

To ease the burden of notation on stratified nilpotent Lie groups $\mathcal G$, we first introduce the curl operator on $\mathbb H^n$ and then establish the div-curl lemma on Heisenberg group  $\mathbb H^n$ (see Theorem \ref{thm2 on H} below).
We note that in \cite{FTT} they also studied the div-curl Lemma on the three dimensional Heisenberg group $\mathbb H$ with respect to the Sobolev space, while the form of their definition of curl operator is closely related to the dimension.
At the end of this section, we provide a suitable definition of the curl operator on stratified nilpotent Lie groups, and point out that the div-curl lemma follows by repeating the proof of that on Heisenberg groups.

\subsection{Curl operator on Heisenberg groups}

We now introduce the curl operator on the Heisenberg group $\mathbb H^n$.
One of the key features of the curl operator on $\R^n$ is that any curl free vector field is
a conservative vector field. Hence, we introduce the curl operator on $\mathbb H^n$, aiming to
preserve this property, which plays an important role in establishing the div-curl lemma.

Note that the vector fields $X_1,\ldots, X_{2n}$ define a vector bundle over $\H^n$ (the horizontal vector bundle ${\bf H}\H^n$) that can be canonically identified with a vector subbundle of the tangent vector bundle of $\R^{2n+1}$. Since each fiber of ${\bf H}\H^n$ can be canonically identified with a vector subspace of $\mathbb R^{2n+1}$, each section $\phi$ of ${\bf H}\H^n$ (called briefly horizontal section) can be identified with a map $\phi : \H^n \to \mathbb R^{2n}$. At each point $p \in \H^n$ the horizontal fiber is indicated as ${\bf H}\H^n_p$ and each fiber can be endowed with the scalar product $\langle\cdot, \cdot\rangle_p$ and the norm $| \cdot |_p$ that make the vector fields $X_1,\ldots,X_{2n}$ orthonormal. Hence we shall also identify a section of ${\bf H}\H^n $ with its canonical coordinates with respect to this moving frame. In this way, a section $\phi$ will be identified with a function $\phi = (\phi_1,\ldots,\phi_{2n}) : \H^n \to \mathbb R^{2n}$. We stress that a horizontal section {$\phi = (\phi_1,\ldots,\phi_{2 n})$}, which can be thought as a multidimensional object, can always be canonically identified with
{the $\mathbb R^{2n}$-valued} function $\phi_1X_1 +\cdots + \phi_{2n}X_{2n}$.

We denote by $C^\infty({\bf H}\H^n)$ the vector space of smooth sections of ${\bf H}\H^n$, and by $\mathcal D({\bf H}\H^n)$ the subset of compactly supported functions in $C^\infty({\bf H}\H^n)$.
Using canonical coordinates, clearly $C^\infty({\bf H}\H^n)$ and $\mathcal D({\bf H}\H^n)$ can be identified respectively with $C^\infty(\H^n)^{2n}$ and $\mathcal D(\H^n)^{2n}$. Analogously the space of horizontal distributions $\mathcal{D}'({\bf H}\H^n)$ can be identified with $\mathcal {D}'(\H^n)^{2n}$.
We denote by $L^2({\bf H}\H^n)$ the space of all measurable sections $\phi=(\phi_1,\ldots,\phi_{2n})$ of ${\bf H}\H^n$ such that $\phi\in L^2(\H^n)^{2n}$.
We also denote by $W^{1,2}(\H^n)$
the set of real-valued functions $f\in L^2(\H^n)$ such that $X_1f,\ldots, X_{2n}f$ belong to $L^2(\H^n)$, endowed with
its natural norm.

Let us recall the horizontal gradient and the horizontal divergence on Heisenberg groups (and similar for general stratified Lie groups)

\begin{definition}
For $f\in \mathcal D'(\H^n)$, the horizontal gradient $\nabla_{\H^n}$ is defined as
$$ \nabla_{\H^n}f:= (X_1f,\ldots, X_{2n}f) $$
and for $\phi= (\phi_1,\ldots,\phi_{2n})\in\mathcal D'({\bf H}\H^n)$, the horizontal divergence is defined as
$$ \operatorname{div}_{\H^n} \phi = \sum_{j=1}^{2n} X_j\phi_j. $$
\end{definition}

We recall that the standard curl operator on $\mathbb R^{2n+1}$ has the form as follows (see for example \cite[Page 507]{D}):
let $\mathcal V=(\mathcal V_1,\ldots, \mathcal V_{2n+1})\in \big( \mathcal D'(\mathbb R^{2n+1})\big)^{2n+1}$, then  ${\rm curl}_{\mathbb R^{2n+1}} \mathcal V$
is a matrix of distributions whose entries are given by
\begin{equation}\label{def curl}
\left\langle ({\rm curl}_{\mathbb R^{2n+1}} \mathcal V)_{i,j}, \phi \right\rangle  =
\left\langle \mathcal V_j, {\partial\phi\over \partial x_i} \right\rangle - \left\langle \mathcal V_i, {\partial\phi\over \partial x_j} \right\rangle 
\end{equation}
for $\phi \in \mathcal D(\mathbb R^{2n+1})$ and $i, j=1,\ldots,2n+1$, where $\mathcal D(\mathbb R^{2n+1})$ is the space of all smooth functions with compact supports and with the usual topology.

Recall that we have the vector fields $X_j$ and $X_{n+j}$($=Y_j$) ($j=1,2,\ldots,n$) as in \eqref{XYT}. Then from the definition of  these {$X_j$'s} we have
\begin{align}\label{H1}
  \begin{bmatrix}
     X_1 \\
     \cdot\\
     \cdot\\
     X_n\\
        X_{n+1} \\
     \cdot\\
     \cdot\\
     X_{2n}\\
     T
     \end{bmatrix}
     =
     \begin{bmatrix}
      I_{2n},\  \vec{J} \\
       \vec O, \  1
     \end{bmatrix}
     \begin{bmatrix}
     {\partial\over \partial x_1} \\
     \cdot\\
     \cdot\\
     {\partial\over \partial x_n} \\
    {\partial\over \partial y_1} \\
     \cdot\\
     \cdot\\
     {\partial\over \partial y_n}\\
     {\partial\over \partial t}
           \end{bmatrix},
\end{align}
where $I_{2n}$ is the identity matrix of order $2n$,
 $\vec O$ is the vector
 $$ \vec O=[0,\ldots,0] $$
 with $2n$ elements,
and
$\vec J$ is the vector
$$ \vec J=[ 2y_1,\ldots,2y_n,-2x_1,\ldots,-2x_n ]^T, $$
here $\vec v^T$ means the transpose of the vector $\vec v$.

We now denote
\begin{align}\label{H2}
A     =
     \begin{bmatrix}
      I_{2n},\  \vec{J} \\
       \vec O, \  1
     \end{bmatrix}.
\end{align}
Then it is direct to see that $A$ is invertible {whose elements are numbers or polynomials}. Then from \eqref{H1} we have
\begin{align}\label{H3n}
          \begin{bmatrix}
     {\partial\over \partial x_1} \\
     \cdot\\
     \cdot\\
     {\partial\over \partial x_n} \\
    {\partial\over \partial y_1} \\
     \cdot\\
     \cdot\\
     {\partial\over \partial y_n}\\
     {\partial\over \partial t}
           \end{bmatrix}=
A^{-1}  \begin{bmatrix}
     X_1 \\
     \cdot\\
     \cdot\\
     X_n\\
        X_{n+1} \\
     \cdot\\
     \cdot\\
     X_{2n}\\
     T
     \end{bmatrix},
\end{align}

We note that from the definition of $X_j$ and $T$, we have
$$ {\partial\over \partial t} = T = {1\over 4}[X_{n+i},X_i] = {1\over 4} (X_{n+i}X_i - X_iX_{n+i})$$
for $i=1, \ldots, n$. The curl operator can be defined as follows.  
\begin{definition}\label{def curl 2}
For every $V=\sum_{j=1}^{2n}V_jX_j\in \mathcal D'({\bf H}\H^n)$, and for a fixed $i=1,\ldots,n$, let
\begin{align}\label{H3}
 \mathcal V=
A^{-1}  \begin{bmatrix}
     V_1 \\
     \cdot\\
     \cdot\\
     V_{2n}\\
     {{1\over 4} (X_{n+i} V_i - X_i V_{n+i})}
     \end{bmatrix}.
\end{align}
Then we define
\begin{align}
\operatorname{curl}_{\H^n} V:= \operatorname{curl}_{\mathbb R^{2n+1}} \mathcal V,
\end{align}
where $\operatorname{curl}_{\mathbb R^{2n+1}} \mathcal V$ is the standard curl operator on $\mathbb R^{2n+1}$.
\end{definition}

We remark that in Definition \ref{def curl 2}, the curl operator is not unique.
At least for each $i=1,\ldots,n$,
we have a corresponding curl operator.
We also remark that from Definition \ref{def curl 2} {and \eqref{H1}}, we have
$$ { {\rm curl}_{\mathbb H^n} \circ \nabla_{\mathbb H^n} = \operatorname{curl}_{\mathbb R^{2n+1}} \circ \nabla_{\mathbb \R^{2 n + 1}} = 0.} $$

The aim that we introduce the curl operator as above is to show that
any curl free vector field is a conservative vector field, which is a well-known result on $\R^n$.
To be more specific, we have the following result.
\begin{proposition}\label{prop curl}
For every $V = (V_1, \cdots, V_{2n}) \in L^2({\bf H}\H^n)$ satisfying $ {\rm curl}_{\H^n} V={\bf 0} $, there exists $\phi\in \mathcal D'(\H^n)$ such that $ V=\nabla_{\H^n}\phi $. Moreover, we can take
\begin{align} \label{NE1}
\phi = - \left(- \sum_{i = 1}^{2 n} X_i^2 \right)^{-\frac{1}{2}} \left\{ \sum_{i = 1}^{2 n} \left[ \left(- \sum_{i = 1}^{2 n} X_i^2 \right)^{-\frac{1}{2}} X_i \right] V_i \right\} \in L^{\frac{2 n + 2}{n}}(\H^n).
\end{align}
\end{proposition}
\begin{proof}
For every $V\in \mathcal D'({\bf H}\H^n)$ satisfying $ \operatorname{curl}_{\H^n} V={\bf 0} $, from Definition \ref{def curl 2}, we obtain that
$$ {\operatorname{curl}_{\R^{2 n + 1}}} \mathcal V ={\bf0}, $$
where $\mathcal V$ is the vector field in $\mathcal D'(\H^n,\mathbb R^{2n+1})$ associated to $V$.

Hence, we obtain that there exists a distribution {$\phi\in \mathcal D'(\R^{2 n + 1}) = \mathcal D'(\H^n)$ such that $\mathcal V = \nabla_{\R^{2 n + 1}}\phi$}. Then, by using the Definition \ref{def curl 2} again we get that
\begin{align}\label{H4}
  \begin{bmatrix}
     V_1 \\
     \cdot\\
     \cdot\\
     \cdot\\
     V_{2n}\\
     {1\over 4} (X_{n+i}V_i - X_iV_{n+i})
     \end{bmatrix} = A \mathcal V = A \,\nabla\phi,
\end{align}
which gives
$$ V_j = {\partial\phi\over \partial x_j} + 2y_j  {\partial\phi\over \partial x_{2n+1}}  $$
and
$$ V_{n+j} = {\partial\phi\over \partial x_{n+j}} - 2x_j  {\partial\phi\over \partial x_{2n+1}} $$
for $j=1,\ldots,n$. By identifying $ {\partial\phi\over \partial x_{2n+1}}$ as $ {\partial\phi\over \partial t}$ we obtain that
$ V_j=X_j\phi $
for $j=1,\ldots,2n$, which implies that
$ V=\nabla_{\H^n}\phi.$ Now, from the $L^2$-boundedness of Riesz transforms and the Hardy-Littlewood-Sobolev inequality, we deduce \eqref{NE1}.
\end{proof}

\subsection{Div-curl lemma on Heisenberg groups}

\begin{theorem}\label{thm2 on H}
Suppose that $E = (E_1, \cdots, E_{2 n})$, $B = (B_1, \cdots, B_{2 n}) \in L^2(\H^n,\R^{2n})$ are vector fields on $\H^n$ taking values in $\R^{2n}$ and satisfy
 $$ \operatorname{div}_{\H^n} E = 0,\quad \operatorname{curl}_{\H^n} B = 0. $$
 Then we have $E\cdot B \in H^1(\H^n)$ with
 $$ \| E\cdot B \|_{H^1(\H^n)}\lesssim\|E\|_{L^2(\H^n,\R^{2n})}\|B\|_{L^2(\H^n,\R^{2n})}. $$
 \end{theorem}
\begin{proof}
Since $E,B\in L^2(\H^n,\R^{2n})$ are vector fields on $\H^n$ taking values in $\R^{2n}$, it is direct that
$ E\cdot B\in L^1(\H^n) $ with
$$  \| E\cdot B \|_{L^1(\H^n)}\leq \|E\|_{L^2(\H^n,\R^{2n})}\|B\|_{L^2(\H^n,\R^{2n})}  $$

Now since $\operatorname{curl}_{\H^n} B=0$, from Proposition \ref{prop curl} above,
there exists $$\psi = - \sum_{i = 1}^{2 n} (-\Delta)^{-\frac{1}{2}} X_i B_i \in L^2(\H^n)$$ such that $B=\nabla_{\H^n} (-\Delta)^{-1/2}\psi =(R_1 \psi, \ldots, R_{2n} \psi)$. Moreover, using the $L^2$-isometry property of Riesz transform, we have
$\|B\|_{L^2(\H^n,\R^{2n})} = \|\psi\|_{L^2(\H^n)}$.

Next, we claim that for $E\in L^2(\H^n,\R^{2n})$ with $ \operatorname{div}_{\H^n} E(g)=0$, we have that
\begin{align}\label{div0}
 \sum_{j=1}^{2n}   R_j^t(E_j)(g)=0.
\end{align}
To see this, we recall that
$$ \operatorname{div}_{\H^n} E = \sum_{j=1}^{2n} X_j E_j$$
and that
$$ R_j^t=-(-\Delta)^{-1/2} X_j. $$
Thus, we have
\begin{align*}
 \sum_{j=1}^{2n}   R_j^t(E_j) = \sum_{j=1}^{2n} -(-\Delta)^{-1/2} X_j E_j = -(-\Delta)^{-1/2}  \sum_{j=1}^{2n} X_j E_j =0,
\end{align*}
which implies \eqref{div0} and the claim holds.

Thus, from \eqref{div0}, we obtain that
\begin{align}
E(g)\cdot B(g) &= \sum_{j=1}^{2n} E_j(g)\, B_j(g)= \sum_{j=1}^{2n} E_j(g)\, R_j(\psi)(g) \nonumber\\
&= \sum_{j=1}^{2n} \bigg( E_j(g)\, R_j(\psi)(g) - \psi(g) R_j^t(E_j)(g) \bigg).\label{EB}
\end{align}

Based on the last equality \eqref{EB} above, it is direct that
$$\int_{\H^n} E(g)\cdot B(g) \,dg=0. $$

Now test the equality \eqref{EB} over all functions in BMO$(\H^n)$, we see that for every $b\in {\rm BMO}(\H^n)$,
\begin{align*}
&\int_{\H^n} E(g)\cdot B(g)\ b(g)dg  \\
&= \sum_{j=1}^{2n}\int_{\H^n} \bigg( E_j(g)\, R_j(\psi)(g) - \psi(g) R_j^t(E_j)(g) \bigg) b(g)dg\\
&= \sum_{j=1}^{2n}\int_{\H^n} [b,R_j](\psi)(g)\,E_j(g)\, dg.
\end{align*}
Since $b\in {\rm BMO}(\H^n)$, from Theorem \ref{thm1} we have that
$$ \|[b,R_j]\|_{L^2(\H^n)\to L^2(\H^n)} \approx \|b\|_{ {\rm BMO}(\H^n)  },  $$
which implies that
\begin{align*}
\bigg|\int_{\H^n} E(g)\cdot B(g)\ b(g)dg  \bigg|
&\lesssim \sum_{j=1}^{2n}  \|b\|_{ {\rm BMO}(\H^n)  } \|E_j\|_{L^2(\H^n)}\,\|\psi\|_{L^2(\H^n)}\\
&\lesssim\|b\|_{ {\rm BMO}(\H^n)  } \|E\|_{L^2(\H^n,\R^{2n})} \|B\|_{L^2(\H^n,\R^{2n})}.
\end{align*}
This yields that
$$ \| E\cdot B \|_{H^1(\H^n)}\lesssim\|E\|_{L^2(\H^n,\R^{2n})}\|B\|_{L^2(\H^n,\R^{2n})}. $$
The proof of Theorem \ref{thm2 on H} is complete.
\end{proof}

\subsection{Div-curl lemma on Lie groups $\mathcal G$ and proof of Theorem \ref{thm2}}

From the proof of Theorem \ref{thm2 on H} in Section 5.2 above, it is clear that it relies on Theorem \ref{thm1} and  the key points  that we used in the proof are Proposition \ref{prop curl} and the fact \eqref{div0}.  We note that Proposition \ref{prop curl} holds since we provide a suitable version of curl operator on the Heisenberg groups $\mathbb H^n$. We also point out that the fact \eqref{div0} is also true on more general settings, such as homogeneous groups.  Thus, to establish a similar result on stratified nilpotent Lie groups, i.e., Theorem \ref{thm2}, it suffices to introduce an appropriate curl operator following the same way as in Section 5.1, such that a similar version of Proposition \ref{prop curl} holds on stratified nilpotent Lie groups.

We now give the curl operator on stratified nilpotent Lie group {$\mathcal G$, by using some elementary notations and properties which can be found in \cite{BLU}.}
We write $\mathcal G = \mathbb R^n\times  \mathbb R^{n_1}\times\cdots\times  \mathbb R^{n_{k-1}}$.
Now set
$$ Z_i(o) ={\partial\over\partial x_i},\quad i=1,\ldots, n $$
and
$$ T_{i,j}(o) =  {\partial\over\partial t_{i,j}},\quad i=1,\ldots, k-1,\quad j=1,\ldots, n_i. $$
Then we have
\begin{align}\label{m1}
  \begin{bmatrix}
     X_1 \\
     \cdot\\
     \cdot\\
     X_n
     \end{bmatrix}
     = A
     \begin{bmatrix}
     Z_1 \\
     \cdot\\
     \cdot\\
     Z_n
     \end{bmatrix},
\end{align}
where $A$ is a $n \times n$ invertible matrix.  Next, we further have
\begin{align}\label{m2}
     \begin{bmatrix}
     Z_1 \\
     \cdot\\
     \cdot\\
     \cdot\\
     Z_n
     \end{bmatrix}
     = \Big( I_n, B \Big)
     \begin{bmatrix}
     {\partial \over \partial x_1} \\
     \cdot\\
     {\partial \over \partial x_n} \\[2pt]
     {\partial \over \partial t_{1,1}} \\
     \cdot\\
     {\partial \over \partial t_{1,n_1}} \\
     \cdot\\
     {\partial \over \partial t_{k-1,1}} \\
    \cdot\\
     {\partial \over \partial t_{k-1,n_{k-1}}},
     \end{bmatrix},
\end{align}
where $B$ is an $n\times \Big(\sum_{i=1}^{k-1} n_i\Big)$ matrix whose elements are polynomials. Now combining
\eqref{m1} and \eqref{m2}, we obtain that
\begin{align}\label{m3}
  \begin{bmatrix}
     X_1 \\
     \cdot\\
     \cdot\\
     X_n
     \end{bmatrix}
     = \Big( A, AB \Big)
     \begin{bmatrix}
     {\partial \over \partial x_1} \\
     \cdot\\
     \cdot\\
    \cdot\\
     {\partial \over \partial t_{k-1,n_{k-1}}}
     \end{bmatrix}.
\end{align}
Moreover, note that we have
\begin{align}\label{m4}
  \begin{bmatrix}
     T_{1,1} \\
     \cdot\\
     T_{1,n_1}\\
     \cdot\\
     \cdot\\
    T_{k-1,1}\\
     \cdot\\
    T_{k-1,n_{k-1}}\\
        \end{bmatrix}
     = \Big( O, \Gamma \Big)
     \begin{bmatrix}
     {\partial \over \partial x_1} \\
     \cdot\\
     \cdot\\
     \cdot\\
     \cdot\\
    \cdot\\
     {\partial \over \partial t_{k-1,n_{k-1}}},
     \end{bmatrix},
\end{align}
where $O$ is a $ \Big(\sum_{i=1}^{k-1} n_i\Big) \times n$ zero matrix
and $\Gamma$ is an upper triangular matrix whose elements in the upper triangular region except the diagonal are polynomials.
Now combining \eqref{m3} and \eqref{m4}, we obtain that
\begin{align}\label{m5}
  \begin{bmatrix}
     X_1 \\
     \cdot\\
     X_n \\
     T_{1,1} \\
     \cdot\\
     T_{1,n_1}\\
     \cdot\\
     \cdot\\
    T_{k-1,1}\\
     \cdot\\
    T_{k-1,n_{k-1}}\\
        \end{bmatrix}
     =   \begin{bmatrix}
     A & AB \\
     O & \Gamma
        \end{bmatrix}
     \begin{bmatrix}
     {\partial \over \partial x_1} \\
     \cdot\\
     \cdot\\
     \cdot\\
     \cdot\\
    \cdot\\
     {\partial \over \partial t_{k-1,n_{k-1}}},
     \end{bmatrix}.
\end{align}
We now point out that
\begin{align}\label{m6}
        \begin{bmatrix}
     A & AB \\
     O & \Gamma
        \end{bmatrix}^{-1}
        =\begin{bmatrix}
     A^{-1} & -B\Gamma^{-1} \\
     O & \Gamma^{-1}
        \end{bmatrix},
     \end{align}
whose elements are numbers or polynomials.
As a consequence, we can obtain the following representation from \eqref{m5} and \eqref{m6}:
\begin{align}\label{m7}
    \begin{bmatrix}
     {\partial \over \partial x_1} \\
     \cdot\\
     \cdot\\
     \cdot\\
     \cdot\\
    \cdot\\
     {\partial \over \partial t_{k-1,n_{k-1}}},
     \end{bmatrix}
          =   \begin{bmatrix}
     A^{-1} & -B\Gamma^{-1} \\
     O & \Gamma^{-1}
        \end{bmatrix}
  \begin{bmatrix}
     X_1 \\
     \cdot\\
     X_n \\
     T_{1,1} \\
     \cdot\\
     T_{1,n_1}\\
     \cdot\\
     \cdot\\
    T_{k-1,1}\\
     \cdot\\
    T_{k-1,n_{k-1}}\\
        \end{bmatrix}.
\end{align}
We also point out that $T_{i,j}$ can be generated by the  Lie bracket of $\{ X_1,\ldots, X_n\}$, $i=1,\ldots,k-1$ and $j=1,\ldots n_i$.
So we fix a representation of $T_{i,j}$ from $\{ X_1,\ldots, X_n\}$ (which may not be unique). Without loss of generality, we
can take
$$ T_{i,j} =\sum a_{(l_1,\ldots, l_i)} X_{l_1}\cdots X_{l_i}.  $$

For any vector field $B=(B_1,\ldots, B_n)$ on $\mathcal G$, we now set
$$ B_{i,j} =  \sum a_{(l_1,\ldots, l_i)} X_{l_1}\cdots X_{l_{i-1}} B_{\ell_i}  $$
for $i=1,\ldots,k-1$ and $j=1,\ldots n_i$.  Then we can define a new vector $\mathcal B$ on $ \mathbb R^{n+n_1+\cdots+n_{k-1}}$  as
\begin{align}\label{new B}
\mathcal B          =   \begin{bmatrix}
     A^{-1} & -B\Gamma^{-1} \\
     O & \Gamma^{-1}
        \end{bmatrix}
  \begin{bmatrix}
     B_1 \\
     \cdot\\
     B_n \\
     B_{1,1} \\
     \cdot\\
     B_{1,n_1}\\
     \cdot\\
     \cdot\\
    B_{k-1,1}\\
     \cdot\\
    B_{k-1,n_{k-1}}\\
        \end{bmatrix}.
\end{align}

\begin{definition}\label{def curl on G}
Let all the notation be the same as above. For any vector field $B=(B_1,\ldots, B_n)$ on $\mathcal G$,
we define
\begin{align}\label{curl G}
\operatorname{curl}_{\mathcal G} B := \operatorname{curl}_{\mathbb R^{n+n_1+\cdots+n_{k-1}}} \mathcal B,
\end{align}
where $\operatorname{curl}_{\mathbb R^{n+n_1+\cdots+n_{k-1}}} \mathcal B$ is the standard curl operator on $\mathbb R^{n+n_1+\cdots+n_{k-1}}$ as pointed out in \eqref{def curl}.
\end{definition}

Then, following the proof of Theorem \ref{thm2 on H}, we can obtain Theorem \ref{thm2}. We omit the details here.





\end{document}